\documentclass[11pt]{article}

\usepackage{amssymb,amsfonts,amsmath,amsthm}
\usepackage{graphicx}

\usepackage{color}
\usepackage{enumerate}
\usepackage{amssymb, amsfonts, amsthm,amsmath}
\usepackage{setspace}
\usepackage{tikz}

\tikzstyle{every node}=[color=black, circle, draw, fill=black, inner sep=0pt, minimum width=2pt]
\tikzstyle{left}=[color=green, circle, draw, fill=green, inner sep=0pt, minimum width=4pt]
\tikzstyle{right}=[color=red, circle, draw, fill=red, inner sep=0pt, minimum width=4pt]
\tikzstyle{s}=[color=black, circle, draw, fill=black, inner sep=0pt, minimum width=1.5pt]
\tikzstyle{l}=[color=black, circle, draw, fill=black, inner sep=0pt, minimum width=3pt]
\tikzstyle{plain}=[color=black, fill=none, draw=none, inner sep=0pt, minimum width=2pt]
\usetikzlibrary{calc}
\usetikzlibrary{decorations.pathreplacing}

\newcommand{\eps}{\varepsilon}
\newcommand{\Pro}[1]{\mathbb{P} \left[\,#1\,\right]}
\newcommand{\Ex}[1]{\mathbb{E} \left[\, #1\,\right]}

\renewcommand{\P}{\mathbb{P}}
\renewcommand{\Pr}{\mathbb{P}}

\newcommand{\V}{V_N}
\newcommand{\D}{\mathcal{D}_R}
\newcommand{\G}{\mathcal{G}}
\newcommand{\Gnan}{\mathcal{G}(N;\alpha,\nu)}
\newcommand{\Pnan}{\mathcal{P}(N;\alpha,\nu)}

\DeclareMathOperator{\core}{Core}
\DeclareMathOperator{\xcore}{XCore}
\DeclareMathOperator{\per}{Per}

\newtheorem{theorem}{Theorem}  
\newtheorem{fact}[theorem]{Fact}
\newtheorem{lemma}[theorem]{Lemma}

\newtheorem{claim}[theorem]{Claim}
\newtheorem{corollary}[theorem]{Corollary}
\newtheorem{remark}[theorem]{Remark}

\newtheorem{definition}[theorem]{Definition}

\numberwithin{theorem}{section}
\numberwithin{equation}{section}

\title{Typical distances in a geometric model for complex networks
}

 \author{Mohammed Amin Abdullah\footnote{Mathematical and Algorithmic Sciences Lab, Huawei Technologies Ltd. 
e-mail: \texttt{mohammed.abdullah@huawei.com}} 
\qquad Michel Bode\footnote{School of Mathematics, University of Birmingham, United Kingdom e-mail: \texttt{michel.bode@gmx.de}} 
\qquad Nikolaos Fountoulakis\footnote{School of Mathematics, University of Birmingham, United Kingdom  e-mail: \texttt{n.fountoulakis@bham.ac.uk}} \thanks{This research has been supported by a Marie Curie Career Integration Grant
 PCIG09-GA2011-293619.}
}

\date{\today}

\begin{document}

\maketitle 
\begin{abstract}
We study typical distances in a geometric random graph on the hyperbolic
plane. Introduced 
by Krioukov et al.~\cite{ar:Krioukov} as a model for complex networks, $N$ vertices are drawn randomly within a bounded subset of the hyperbolic plane 
and any two of them are joined if they are within a threshold hyperbolic distance. 
With appropriately chosen parameters, the random graph is sparse and exhibits power law degree distribution 
as well as local clustering.  In this paper we show a further property:  the distance between two uniformly chosen vertices that belong to the same component is doubly logarithmic in $N$, i.e., the graph is an  ~\emph{ultra-small world}. More precisely, we show that the distance rescaled by $\log \log N$ converges in probability to a certain constant that depends 
on the exponent of the power law. The same constant emerges in an analogous setting with the well-known \emph{Chung-Lu} model for which the degree distribution has a power law tail.

\end{abstract}
\section{Introduction} 

The~\emph{small-world problem} was first stated by Stanley Milgram in his 1967 paper~\cite{ar:Milgram1967} through which he  gave 
strong evidence of the so-called \emph{small-world effect}. The simplest formulation of the small-world problem~\cite{ar:Milgram1967} 
is: ``Starting with any two people in the world, what is the probability that they will know each other?".
A more sophisticated formulation of the problem asks whether any two people, if they do not directly know of each other,  have common 
acquaintances. Milgram's experiment indicated that this is indeed the case within a relatively small random sample of the population of the 
United States. In particular, it turned out the at least half of the sample was within \emph{six degrees of separation} from the 
``target" individual. In graph theoretic terms, in the graph of acquaintances the nodes that represent these individuals are within 
distance 6 from the node that was representing the target individual. 

The small-world phenomenon is ubiquitous in natural and technological networks such as neural networks, 
the Internet, the World-Wide-Web or the power grid -- see the book of Chung and Lu~\cite{ChungLuBook+} as well as the book 
of Dorogovtsev~\cite{Dor} for experimental evidence regarding such networks. For example, it was announced relatively recently 
that between any two active users of Facebook there are 3.74 degrees of separation on average~\cite{url:Facebook}. 

There have been numerous attempts to explain this phenomenon through the theory of complex networks. Among the initial attempts 
was the ``small-world" model of Watts and Strogatz which is defined through random re-wiring of the edges of a cyclic lattice. This 
model exhibits small average distance, but lacks a basic feature of such large self-organizing networks which is the \emph{scale freeness}. 
Experimental evidence~\cite{AlbBar} suggests that these networks have a distribution of degrees whose tail decays like a power law with 
exponent usually between 2 and 3. 

Of course the term ``small-world" itself is somewhat vague. Loosely speaking, the term refers to average distances that are slowly growing 
functions of the number of vertices of the network. A possible candidate is the logarithmic function. 
Thus, the classical Erd\H{o}s-R\'enyi random graph may be thought of as a small-world 
graph as it has logarithmic diameter -- see~\cite{Bol01}. However, it lacks the scale freeness as well and, furthermore, it represents a 
very homogeneous network. This is a very unrealistic feature as most large scale networks contain vertices that have very different
properties from each other. Sub-logarithmic bounds on the diameter were established for the preferential attachment model~\cite{BarAlb}
by Bollob\'as and Riordan~\cite{Diam}. As it was shown by Bollob\'as et al.~\cite{DegSeq}, this is scale-free with exponent equal to 3. 

Recent research that focused on models for complex networks that are scale free with power law exponent between 2 and 3 
identified cases of such networks that are \emph{ultrasmall}. This term is associated with models in which the distance between 
two randomly chosen connected vertices grows \emph{doubly logarithmically} in the number of vertices of the random graph. 
With $N$ denoting the number of vertices, the function $\log \log N$ is a very slowly growing function. Presumably this is closer 
to empirical evidence which comes from networks that have millions of vertices but whose average distance between two randomly 
chosen vertices is very small. 

An analytical relation between the two was first established by Cohen and Havlin~\cite{ar:CohenHavlin} and by Dorogovtsev, Mendes and 
Samukhin~\cite{ar:DorMendSam}. It was shown rigorously for a variety of random graph models which exhibit power law degree distribution 
such as the Chung-Lu model~\cite{ChungLu1+}, the Norros-Reittu model~\cite{ar:NorReit}, the configuration model~\cite{ar:vdHHogm}
as well as variations of the preferential attachment model~\cite{ar:DomvdHofHoog}~\cite{ar:DerMoenMoert}.

\subsection{A geometric framework for complex networks}
Recently, Krioukov et al.~\cite{ar:Krioukov} introduced a geometric framework in order to describe the inherent inhomogeneity
of a complex network. Their basic assumption is that the intrinsic hierarchies 
that are present in a complex network induce a tree-like structure. This suggests that the geometry of a complex network 
is hyperbolic. 

There are several representations of the hyperbolic plane. In this paper, we shall use the 
Poincar\'e unit disk representation, which is simply the open disk of radius one, that is, 
$\{(u,v) \in \mathbb{R}^2 \ : \ 1-u^2-v^2 > 0 \}$, 
which is equipped with the hyperbolic metric: ${4}~{du^2 + dv^2\over (1-u^2-v^2)^2}$. 
This is the standard formulation of the hyperbolic plane. 


In particular, a suitable integration of the metric shows that 
the length of a circle of radius $r$ (centered at the origin)  is
$2\pi~\sinh (r)$, whereas the area of this circle (centered at the origin) is
$2\pi (\cosh ( r) - 1)$. Hence, a fundamental difference with the Euclidean plane is that 
volumes grow exponentially. 

We are now ready to give the definition of the basic model introduced in~\cite{ar:Krioukov}.
Consider the Poincar\'e disk representation of the hyperbolic plane $H^{2}_{-1}$. 
Let $N$ be the number of vertices of the random graph, and we assume that $N\to \infty$. 
Consider also some fixed constant $\nu >0$ and let $R>0$ satisfy $N= \nu e^{R /2}$.
It turns out that the parameter $\nu$ determines the average degree of the random graph. 
Consider the disk $\D$ of hyperbolic radius $R$ centered at the origin of the Poincar\'e disk (that is, the set of points of the Poincar\'e disk at hyperbolic distance at most $R$ from its origin).

Let $\V := \{ v_1,\ldots, v_N \}$ be a binomial point process on $\D$. 
This is a random set of points of size $N$ that are the outcomes of the $i.i.d.$ random variables $v_1,\ldots , v_N$ taking values on $\D$.  
(We will be referring to the random variables $v_i$ as vertices, meaning their outcomes on $\D$.)
More specifically,  assume that $v_1$ has \emph{polar} coordinates $(r, \theta)$. The angle $\theta$ is uniformly distributed in $(0,2\pi]$ and the probability density function of
$r$, which we denote by $\rho_N (r)$, is determined by a parameter $\alpha >0$ and is equal to
\begin{equation} \label{eq:pdf}
 \rho(r) = \rho_N (r) = \begin{cases}
\alpha {\sinh  \alpha r \over \cosh (\alpha R ) - 1}, & \mbox{if $0\leq r \leq R$} \\
0, & \mbox{otherwise}
\end{cases}.
\end{equation}
The aforementioned formulae for the area and the length of a circle of a given radius imply that if we set 
$\alpha =1$, the distribution described in~\eqref{eq:pdf} is the uniform distribution on $\D$ (under the hyperbolic metric). 
For general $\alpha > 0$ Krioukov et al.~\cite{ar:Krioukov} called this the \emph{quasi-uniform} distribution on $\D$. 
Let us remark that in fact  this is the uniform distribution on a disc of hyperbolic radius $R$ within $H^2_{-\alpha^2}$ (the hyperbolic plane that has curvature $-\alpha^2$).  


Given the point process $\V$ on $\D \subset H_{-1}^2$ and the fixed parameters $\alpha$ and $\nu$ we define the random graph $\G (N; \alpha, \nu)$  on the point-set of $\V$, where two distinct points
form an edge if and only if they are within (hyperbolic) distance $R$ from each other. Figure~\ref{fig:tube} shows the 
disc of radius $R$ around a point $p \in \D$. Thus, any point that falls 
inside this disc becomes connected to $p$. 
\begin{figure}[h]
\centering
\includegraphics[scale=0.25]{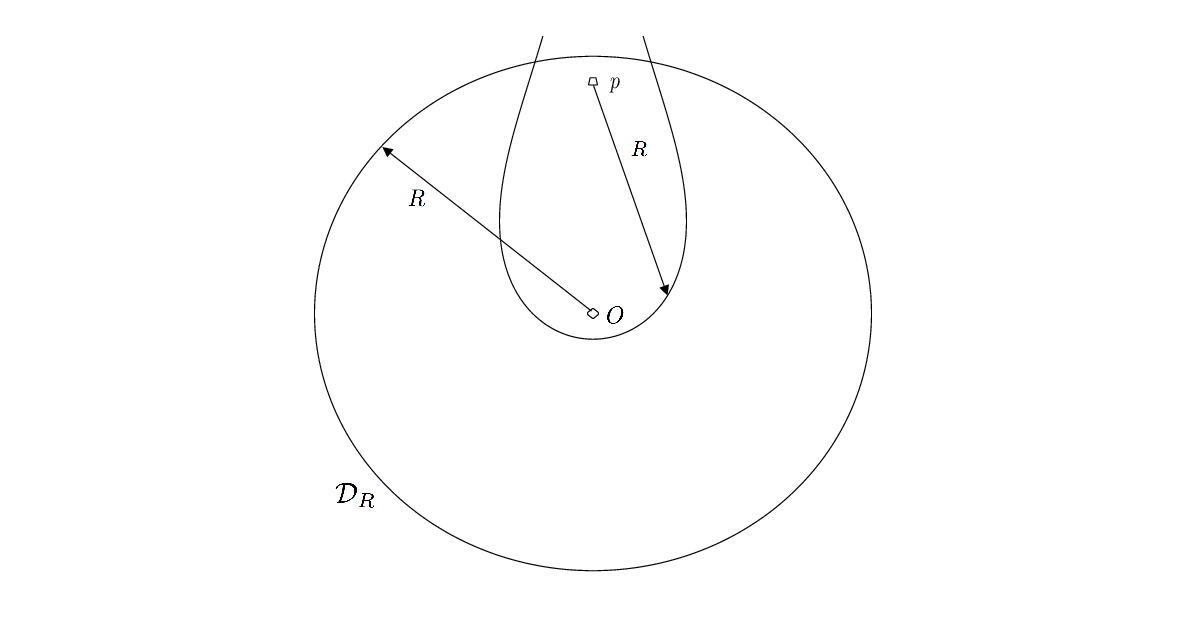}
 \caption{The disc of radius $R$ around point $p$.}
 \label{fig:tube}
 \end{figure}

Figure~\ref{fig:simulation} depicts results of a simulation of the $\Gnan$ model, for $N=1000$, $\nu =3$ and $\alpha = 1.8, 1$ and $0.7$, respectively. Observe now the change in the structure of the random graph as $\alpha$ crosses 1. 
We will comment on this in Section~\ref{sec:geom_asp}.
\begin{figure}[h]
\centering
\includegraphics[scale=0.2]{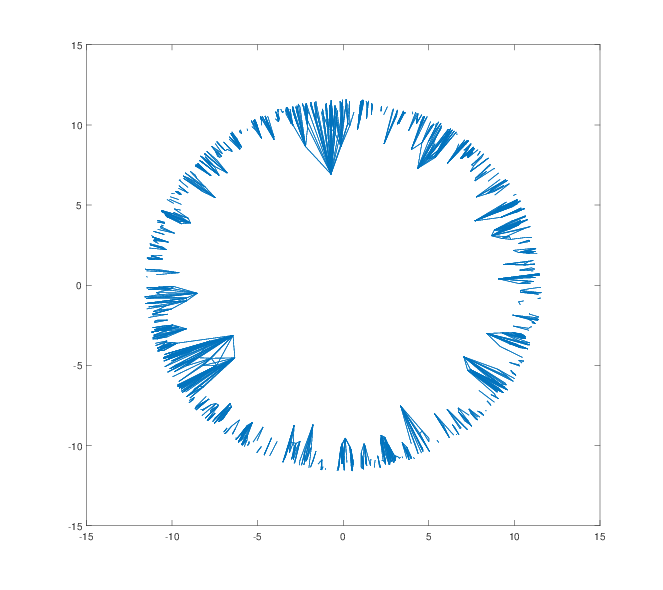}
\includegraphics[scale=0.2]{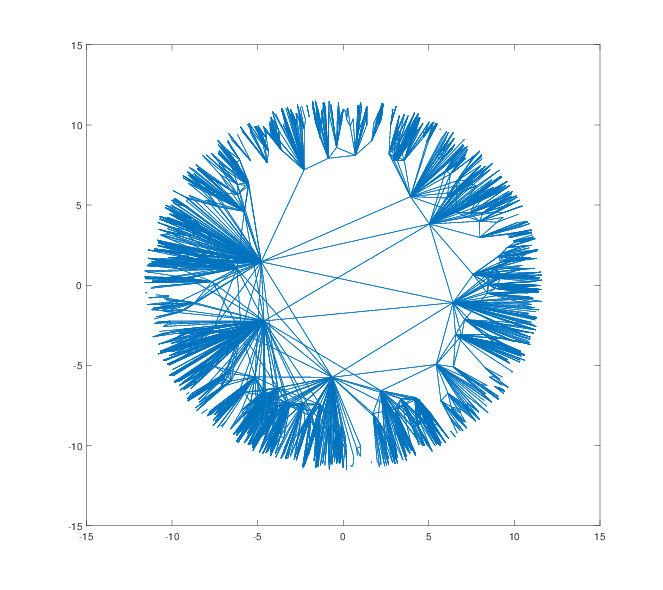}
\includegraphics[scale=0.2]{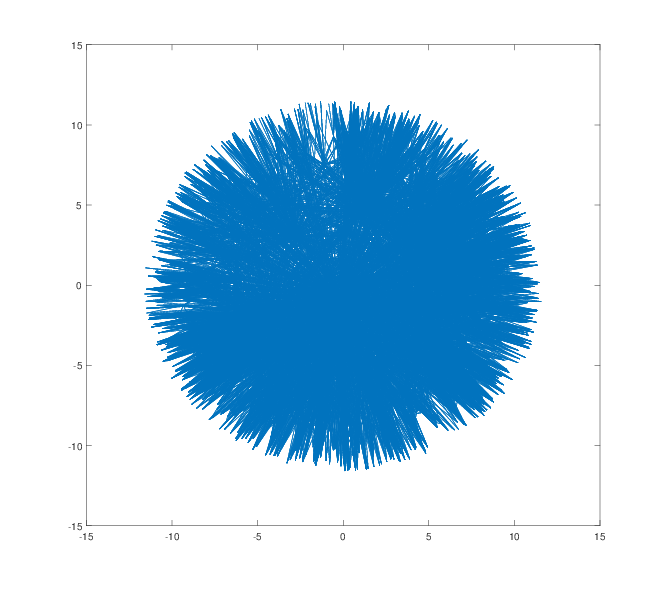}
 \caption{Instances for the $\Gnan$ for $N=1000$, $\nu =3$ and $\alpha =1.8, 1$ and $0.7$.}
 \label{fig:simulation}
 \end{figure}

\subsubsection*{Notation} We now introduce some notation which we use throughout out proofs. 
Let $a_N, b_N$ be two sequences of positive real numbers. We write $a_N \approx b_N$ to indicate that 
$a_N = \Theta (b_N)$, that is, there are real numbers $c,C>0$ such that $c b_N \leq a_N \leq C b_N$, for all natural numbers $N$. 
We also write $a_N \sim b_N$ to denote that $a_N / b_N \rightarrow \infty$, as $N\rightarrow \infty$. 

If $\mathcal{E}_N$ is an event on the probability space  $(\Omega_N, \P_N, \mathcal{F}_N)$, for each $N \in \mathbb{N}$,
we say that $\mathcal{E}_N$ \emph{occurs asymptotically almost surely (a.a.s.)} if $\P (\mathcal{E}_N) \rightarrow 1$ as 
$N \rightarrow \infty$. In our context, we mainly use the sequence of probability spaces that is induced by $\G(N;\alpha, \nu)$. 
However, later we introduce a variant of this model which is its \emph{Poissonisation}. We will be using the term \emph{a.a.s.}
for that model as well. 


\subsubsection{Some facts about $\G (N; \alpha, \nu)$} \label{sec:geom_asp}

We argue that the above model can be thought of as a geometrization of the random graph model that was introduced by 
F.~Chung and L.~Lu~\cite{ChungLu1+, ChungLuComp+} and is a special case of an inhomogeneous random graph. 
The notion of \emph{inhomogeneous random graphs} was introduced by S\"oderberg~\cite{ar:s02}, but
was defined more generally and studied in great detail by Bollob\'as, Janson and Riordan in~\cite{BJR}.
In its most general setting, there is an underlying compact metric space $\mathcal{S}$ equipped with a measure $\mu$ on its 
Borel $\sigma$-algebra. This is the space of \emph{types} of the vertices (defined below). A \emph{kernel} $\kappa$ is a bounded
real-valued, non-negative function on $\mathcal{S} \times \mathcal{S}$, which is symmetric and measurable. 
The vertices of the random graph can be understood as points in $\mathcal{S}$. 
If $x, y \in \mathcal{S}$, then the corresponding vertices are joined with probability 
${\kappa (x,y) \over N} \wedge 1$, independently of every other pair ($N$ is the total number of vertices). 
The points that are the vertices of the graph are approximately distributed according to $\mu$. 
More specifically, the empirical distribution function on the $N$ points converges weakly to $\mu$ as $N \rightarrow \infty$. 

Of particular interest is the case where the kernel function can be factorized and can be written $\kappa (x,y) = t(x)t(y)$; this is called a \emph{kernel of rank 1}. 
Intuitively, the function $t(x)$ can be thought of as the weight or the type of vertex $x$. It is approximately its expected
degree. In the special case where $t(x)$ follows a distribution that has a power law tail, the model becomes the so-called 
Chung-Lu model that was introduced in a series of papers~\cite{ChungLu1+, ChungLuComp+} (see also~\cite{bk:vdH}). 

We argue that in the random graph $\G (N; \alpha, \nu)$, the probability that two vertices are adjacent has this form.  
The proof of this fact relies on Lemma~\ref{lem:relAngle}, which we will state later and is proved in \cite{bode_fountoulakis_mueller}. 
It provides an approximate characterization of what it means for two points $u,v$ to
have hyperbolic distance at most $R$ in terms of their \emph{relative angle}, which we denote by $\theta_{u,v}$. 
This is defined as $\min \{ \hat{uOv}, 2 \pi - \hat{uOv} \}$. Note that $\theta_{u,v}\leq \pi$.  
For this lemma, we need the notion of the \emph{type} of a vertex. 
For a vertex $v \in \V$, if $r_v$ is the distance of  $v$ from the origin, that is, the radius of $v$, then we set $t_v = R - r_v$ -- 
we call this quantity the \emph{type} of vertex $v$. 
As we shall shortly see, the type of a vertex is approximately exponentially distributed. If we substitute $R-t$ for $r$ in 
(\ref{eq:pdf}), then assuming that $t$ is fixed that expression becomes asymptotically (as $N\to \infty$) equal to $\alpha e^{-\alpha t}$. 
Roughly speaking, Lemma~\ref{lem:relAngle} states that two vertices $u$ and $v$ of types $t_u$ and $t_v$ are within distance $R$
(essentially) if and only if $\theta_{u,v} < 2 \nu {e^{t_u/2} e^{t_v/2}}/N$. Hence, conditional on their types the probability that $u$ and
$v$ are adjacent is proportional to ${e^{t_u/2} e^{t_v/2}}/N$. If we set $t(u)= e^{t_u/2}$, then $\P (t(u) \geq x) = \P (t_u \geq 2 
\ln x)\approx e^{-2\alpha \ln x} = 1/x^{2\alpha}$. In other words, the distribution of $t(u)$ has a power-law tail with parameter $2\alpha$. 
Thus, the random graph $\G (N; \alpha, \nu)$ is a \emph{dependent} version of the Chung-Lu model that emerges naturally from the 
hyperbolic geometry of the underlying space. The fact that this is a random geometric graph gives rise to local 
clustering, which is missing in the Chung-Lu model. There, most vertices have tree-like neighborhoods.  

In fact, it can be shown that the degree of a vertex $u$ in $\G (N; \alpha, \nu)$ that has type $t_u$ is approximately distributed as 
a Poisson random variable with parameter proportional to $e^{t_u/2}$. 

Gugelmann, Panagiotou and Peter~\cite{ar:Kosta} showed that the degree of a vertex has a power law with exponent $2\alpha  +1$.
If $\alpha  > 1/2$, then the exponent of the power law may take any value greater than 2. 
When $1 > \alpha >  1/ 2$, this exponent is between 2 and 3.
They also showed that the average degree is a constant that depends on $\alpha$ and $\nu$,
and that the clustering coefficient (the probability of two vertices with a common
neighbor to be joined by an edge) of $\G (N;\alpha, \nu)$ is asymptotically  bounded away from $0$ with probability $1-o(1)$ as
$N\rightarrow \infty$. 

Furthermore, the last two authors together with M\"uller~\cite{bode_fountoulakis_mueller} showed that $\G (N;\alpha, \nu)$ with high
probability has a \emph{giant component}, that is, a connected component containing a linear number of
vertices if $1/2 <\alpha < 1$. When $\alpha > 1$, the size of the largest component is bounded by a function that is \emph{sublinear} in
$N$. This transition is also indicated in Figure~\ref{fig:simulation}. 
When $\alpha =1$, the existence of a giant component depends on the value of $\nu$. There is a critical value for $\nu$ around which a giant component emerges. 

\subsection{Results} 
In this contribution, we give an almost sure bound on the (graph) distance between two randomly chosen vertices that belong to 
the same connected component. 
We show that $\G (N;\alpha, \nu)$ is \emph{ultrasmall} when $\frac{1}{2} < \alpha < 1$, that is, when the degree distribution 
has a power law tail with exponent between  2 and 3. More specifically, we show that a.a.s. the graph distance between two randomly 
chosen vertices that belong to the same component is of order $\log \log N$. However, the diameter of $\G (N; \alpha, \nu)$ grows at least
logarithmically in $N$. This is a recent result of Kiwi and Mitsche~\cite{KiwiMitsche}, where they show that there is a connected component 
of diameter proportional to $\log N$.  They also derive an upper bound on the diameter showing that the diameter is at most proportional to $R^{1+C}$ a.a.s., 
for some positive constant $C$ that depends on the parameters of the model. More recently, Friedrich and Krohmer~\cite{FriedKroh15} 
improved the constant showing that the exponent is at most $1/(2(1-\alpha))$. They also show that if $\nu$ is small enough, then 
the exponent is equal to 1. 
Note that the Chung-Lu model exhibits logarithmic diameter~\cite{ChungLu1+}.

For $\alpha > 1$, we show that a.a.s. $\G (N;\alpha,\nu)$ is \emph{almost} ultrasmall: the graph distance between two 
randomly chosen vertices that belong to the same component is a.a.s. bounded by some polynomial of $\log \log N$.   
This range of $\alpha$ yields a power law degree distribution with exponent greater than 3. For this range, 
Chung and Lu~\cite{ChungLu1+} proved that the Chung-Lu model exhibits average distances of order $\log N$ asymptotically with 
high probability.  

Let $d_G(u,v)$ denote the graph distance between two vertices 
$u, v \in \V$. 
\begin{theorem} \label{thm:ultrasmall} 
Assume that $1/2 < \alpha < 1$ and 
let $\tau$ be such that $\tau^{-1} = \log \left( \frac{1}{2\alpha -1} \right)$. Let $u,v \in \V$ be a pair of distinct 
vertices chosen uniformly at random. 
For any $\zeta >0$, the following holds a.a.s.:  either $d_G (u,v) = \infty$ or $\left|\frac{d_G(u,v)}{\log R}-2\tau\right|<\zeta$. 
\end{theorem}
In this regime, $\Gnan$ does have a giant component and therefore for any two distinct vertices $u,v$ we have 
$d_G(u,v)<\infty$ with probability that is asymptotically bounded away from 0. Moreover, Kiwi and Mitsche \cite{KiwiMitsche} showed that the second largest component a.a.s. contains $o(N)$ vertices. 
Hence, if we considert two distinct vertices $u$ and $v$ from $\V$, a.a.s. either they belong to different 
components or they belong to the largest component. 
These facts together with the above theorem imply that if
the vertices $u$ and $v$ are selected uniformly at random 
from the largest component, then 
for any $\zeta >0$,  a.a.s. 
$\left|\frac{d_G(u,v)}{\log R}-2\tau\right|<\zeta$.

The upper bound (which is probably the most important) in the above result was also derived by Chung and Lu~\cite{ChungLu1+} 
for the Chung-Lu model with power law exponent between $2$ and $3$. That was under the assumption that the average degree 
is greater than 1. However, in our case a giant component is formed \emph{independently} of what the average degree is, as long
as $1/2 < \alpha <1$.  The full result for the Chung-Lu model can be found in~\cite{bk:vdH}.
The analogous result but in a stronger form which involves convergence in distribution was obtained 
by van der Hofstad et al.~\cite{ar:vdHHogm} . 



Our second result provides an upper bound on the typical distance between two connected vertices when $\alpha >1$. 
In this case there is no giant component a.a.s. However, the largest component contains polynomially\footnote{Note that by ``polynomial'' 
we mean a function of the form $f(x)=x^c$ where $c>0$ is a constant.} (in $N$) many vertices,
as there exist vertices of degree that scales polynomially in $N$ (and, of course, the neighbours of a vertex all belong to the same component). However, these components form also (almost) 
\emph{ultrasmall} worlds. 
\begin{theorem} \label{thm:almostultrasmall}
Let $\alpha>1$, $\eps>0$.
A.a.s. there is a subset $V'$ of vertices of $\Gnan$ of size $(1-o(1))N$ so that if $u,v\in V'$ and $d_G(u,v) < \infty$, then
$d_G(u,v) \leq \log^{1+\eps}\log N$.
\end{theorem}
We believe that the above theorem can be strengthened in the following sense.  A.a.s. there is a ``representative'' set $V'$ (in the sense of the above theorem) such that for any $u,v \in V'$ we have 
$d_G(u,v)\leq C \log \log N$, for some constant $C>0$, that depends on the parameters of the model. More specifically, we believe that a typical vertex in a connected component has an ultra-short path to a ``hub'' vertex of the component, that is, a vertex of polynomially growing degree.

In the next section, we introduce the \emph{Poissonisation} of $\Gnan$ which is convenient for our calculations. Thereafter, we will 
state and prove some basic geometric facts regarding the hyperbolic plane, which allow us to express distances on the hyperbolic 
plane in terms of polar coordinates on $\D$. Subsequently, we proceed with the proof of Theorems~\ref{thm:ultrasmall} 
and~\ref{thm:almostultrasmall}. 

The main idea behind the proof of Theorem~\ref{thm:ultrasmall} makes use of the existence of a very dense core that is formed 
by those vertices that have type at least $R/2$. We show that if two vertices are connected, then most likely they have short paths to 
the core which itself is a complete graph. These paths, which we call \emph{exploding}, emerge also in the Chung-Lu model
~\cite{ChungLu1+,bk:vdH}.
 
\section{Preliminary results}

\subsection{Poissonisation}
Recall that $\D$ is the disk of hyperbolic radius $R$ around the origin $O$ within the Poincar\'e disk representation of the hyperbolic plane with curvature $-1$. 
It will be significantly easier to work in a setting where, instead of having exactly $N$ random points, the vertex set
is the point-set of a Poisson point process on $\D$ with intensity 
$$N \frac{1}{2\pi}\rho_N (r) dr d\theta .$$

 Two vertices/points are declared adjacent exactly as in $\Gnan$. 
We denote the resulting graph by $\Pnan$.
More specifically, the vertex set consists of the points of the above Poisson point process in $\D$ (see~\cite{Kingman}). 
In every measurable set $U\subseteq \D$, the number of points in $U$ follows the Poisson distribution with parameter equal to 
$N  \mu_\alpha (U)$ where we define

\begin{equation}
\mu_\alpha (U):= \frac{1}{2\pi}\int_U \rho_N (r) dr d\theta.\label{mu(U)Definition}
\end{equation}

Moreover, the numbers of points in any finite collection of pairwise disjoint measurable subsets of $\D$ are
independent Poisson-distributed random variables. 

Furthermore, several of our results concern the $\G (N; \alpha, \nu)$ distribution conditional on a certain subset of points being equal to
$X$. Of course, such an event has probability equal to 0, and thus, to make sense of this one defines the \emph{Palm distribution}. This was first defined by Palm~\cite{ar:palm} in the context of stationary point processes on the real line - see the article of Jagers~\cite{ar:jaegers} and the references therein for a rigorous exposition of this concept.
Let us begin with the $\Pnan$ model. 
For a certain measurable subset $U \subset \D$ such that $\D\setminus U$
has positive Lebesgue measure, the vertex set of the random graph $\mathcal{P}_{X,U}(N;\alpha, \nu)$ consists of $X \subset \D \setminus U$
together with set of points of a Poisson process on 
$\D \setminus U$  with intensity 
$$(N-|X| )\frac{\rho_N (r)}{\mu_\alpha (\D \setminus U)} dr d\theta .$$
(Here and later we assume that $N > |X|$.) 
Hence, this process ``produces" $N-|X|$ points on average
and, therefore, the random graph has $N$ vertices in total, 
on average.

Similarly, for the $\Gnan$ model, its corresponding Palm distribution is 
defined through the $\mathcal{G}_{X,U}(N;\alpha, \nu)$ model. 
The vertex set of this random graph is 
the set $X \subset \D\setminus U$ together with the set of points that are the outcomes of the random variables $v_{|X|+1},\ldots, v_N $ of the set $V_N$, conditional on not being in the set $U$. Two points are adjacent as in the $\Gnan$ model. 
Let us observe that in the particular case where 
$X, U= \emptyset$, the resulting random graph is distributed as $\mathcal{G}(N;\alpha, \nu)$. 

 
Let $X$ be a set of points in $\D$ and let $U\subset \D$ be a measurable set such that $X\cap U=\emptyset$. Let 
$\mathcal{A}_{X}$ be an event in the probability space of the random graph $\mathcal{P}_{X,U}(N;\alpha, \nu)$.  
We call  $\mathcal{A}_{X}$
\emph{non-decreasing} if 
$$ \P_{\mathcal{P}_{X,U}(N;\alpha, \nu)} ( \mathcal{A}_X \ | \ \mathrm{Po}(N-|X|)=N_1) \leq  
\P_{\mathcal{P}_{X,U}(N;\alpha, \nu)} ( \mathcal{A}_X \ | \ \mathrm{Po}(N-|X|)=N_2), $$
whenever $N_1 \leq N_2$. 
If the opposite inequality holds, we call the property \emph{non-increasing}. 


\begin{lemma}\label{lem:poisson_cond}
If $\mathcal{A}_X$ is a non-decreasing event that is associated with a certain set of vertices $X$ and a measurable $U\subset \D$ such that $X\cap U=\emptyset$,
where $\D\setminus U$ has positive Lebesgue measure.

For any $N$ that is large enough
$$ 
\P_{\mathcal{P}_{X,U}(N;\alpha, \nu)} (\mathcal{A}_X ) \geq \frac{1}{4} \P_{\mathcal{G}_{X,U}(N;\alpha, \nu)} ( \mathcal{A}_X ). 
$$
 The same holds, if $\mathcal{A}_X$ is non-increasing and $N$ is large enough with respect to $|X|$.
\end{lemma}

\begin{proof}
Suppose first that $\mathcal{A}_X$ is non-decreasing. We have
\begin{eqnarray*}
\lefteqn{\P_{\mathcal{P}_{X,U}(N;\alpha, \nu)} (\mathcal{A}_X ) =} \\
&   &\sum_{N'=0}^\infty \P_{\mathcal{P}_{X,U}(N;\alpha, \nu)} ( \mathcal{A}_X \ | \ \mathrm{Po}(N-|X|)=N')\P(\mathrm{Po}(N)=N')\\
&\geq& \sum_{N'=N-|X|}^\infty \P_{\mathcal{P}_{X,U}(N;\alpha, \nu)} ( \mathcal{A}_X \ | \ \mathrm{Po}(N-|X|)=N')\P(\mathrm{Po}(N)=N')\\
&\geq& \sum_{N'=N-|X|}^\infty \P_{\mathcal{P}_{X,U}(N;\alpha, \nu)} ( \mathcal{A}_X \ | \ \mathrm{Po}(N-|X|)=N-|X|)\P(\mathrm{Po}(N)=N').
\end{eqnarray*}

Observe that $$\P_{\mathcal{P}_{X,U}(N;\alpha, \nu)} ( \mathcal{A}_X \ | \ \mathrm{Po}(N-|X|)=N-|X|)=\P_{\mathcal{G}_{f(S,X),U}(N;\alpha, \nu)} ( \mathcal{A}_X )$$

Thus,
\begin{eqnarray*}
\P_{\mathcal{P}_{X,U}(N;\alpha, \nu)} (\mathcal{A}_X )
&\geq& \sum_{N'=N-|X|}^\infty \P_{\mathcal{G}_{X,U}(N;\alpha, \nu)} ( \mathcal{A}_X )\P(\mathrm{Po}(N)=N')\\
&=&\P_{\mathcal{G}_{X,U}(N;\alpha, \nu)} ( \mathcal{A}_X )\P(\mathrm{Po}(N)\geq N-|X|)\\
&>&\frac{1}{4}\P_{\mathcal{G}_{X,U}(N;\alpha, \nu)}( \mathcal{A}_X ).
\end{eqnarray*}
where the last line holds $N$ large enough (by an application of, say, the central limit theorem).

For $\mathcal{A}_X$ non-increasing, the proof is similar, bounding the sum by taking only the terms where $N'\leq N-|X|$:
\begin{eqnarray*}
\lefteqn{\P_{\mathcal{P}_{X,U}(N;\alpha, \nu)} (\mathcal{A}_X ) =}\\
&  &\sum_{N'=0}^\infty \P_{\mathcal{P}_{X,U}(N;\alpha, \nu)} ( \mathcal{A}_X \ | \ \mathrm{Po}(N-|X|)=N')\P(\mathrm{Po}(N)=N')\\
&\geq& \sum_{N'=0}^{N-|X|} \P_{\mathcal{P}_{X,U}(N;\alpha, \nu)} ( \mathcal{A}_X \ | \ \mathrm{Po}(N-|X|)=N')\P(\mathrm{Po}(N)=N')\\
&\geq& \sum_{N'=0}^{N-|X|} \P_{\mathcal{P}_{X,U}(N;\alpha, \nu)} ( \mathcal{A}_X \ | \ \mathrm{Po}(N-|X|)=N-|X|)\P(\mathrm{Po}(N)=N')\\
&=&\sum_{N'=0}^{N-|X|} \P_{\mathcal{G}_{X,U}(N;\alpha, \nu)} ( \mathcal{A}_X )\P(\mathrm{Po}(N)=N')\\
&=&\P_{\mathcal{G}_{X,U}(N;\alpha, \nu)} ( \mathcal{A}_X )\P(\mathrm{Po}(N)\leq N-|X|)\\
&>&\frac{1}{4}\P_{\mathcal{G}_{X,U}(N;\alpha, \nu)}( \mathcal{A}_X ),
\end{eqnarray*}
if $N$ is large enough with respect to $|X|$. 
\end{proof}

Let $\mathcal{A}_n$ denote a set of graphs on $\{1,\ldots, n\}$ that is closed under automorphisms.
We call a family $\mathcal{A}=\{\mathcal{A}_n\}_{n \in \mathbb{N}}$ of graphs \emph{(vertex-) non-decreasing}, if $G-v \in \mathcal{A}_{n-1}$ for any\footnote{$G-v \in \mathcal{A}_{n-1}$ means that $G-v$ is isomorphic to a member of $\mathcal{A}_{n-1}$}
$v\in V(G)$ implies $G\in\mathcal{A}_n$.
Similarly, we call the family \emph{(vertex-) non-increasing}, if $G-v\notin \mathcal{A}_{n-1}$ for any $v\in V(G)$ implies
$G\notin\mathcal{A}_n$.

We interpret $\Pnan\notin\mathcal{A}$ as follows. 
Let $n'$ denote the random variable that is the number 
of  points of the Poisson process. If we label these by the 
numbers $\{1, \ldots, |V(\Pnan)| \}$ and consider the graph 
with this vertex set, then this belongs to $\mathcal{A}_{|V (\Pnan)|}$.

We assume that the points that form vertex set of $\Gnan$ inherit the
labeling from the random variables $\{v_1,\ldots v_N\}$. That is, the 
point that is the outcome of $v_i$ is labeled $i$.  
Hence, writing $\Gnan\notin\mathcal{A}$ means that 
the graph obtained by relabeling the vertices of $\Gnan$ by
$\{1,\ldots, N\}$ as above belongs to $\mathcal{A}_N$.

Let us observe that setting $X=\emptyset$ and $U=\emptyset$ and applying Lemma \ref{lem:poisson_cond}, we can say for $\mathcal{A}$ non-decreasing or non-increasing, $\P(\Pnan\notin\mathcal{A})=o(1)$ implies $\P(\Gnan\notin\mathcal{A})=o(1)$.
The results above will allow us to transfer results from the Poisson model into the $\Gnan$ model.

Finally, the following useful fact follows directly from the definition of the process, using the measure defined for the distribution of the points.

\begin{fact}\label{fct:empty}
	Let $A$ be a subset of $\D\setminus U$, for some measurable subset $U \subset \D$, and 
    $X$ be a set of vertices located in $\D$, such that $X\cap (U \cup  A) = \emptyset$.  
	Let $N_A$ be the expected number of vertices in $A$, in $\mathcal{G}_{X,U}(N;\alpha, \nu)$, and denote by $\mathcal{E}_A$ the event
    that $A$ is empty.
	We have
	\[\P_{\mathcal{P}_{X,U}(N;\alpha, \nu)}(\mathcal{E}_A)=\exp(-N_A).\]
\end{fact}

In order to avoid overloading, we will make a convention when 
passing to the Poisson model $\mathcal{P}_{X,U} (N;\alpha, \nu)$. Whenever we deal 
with the Palm distribution of $\mathcal{G} (N;\alpha, \nu)$,
conditioning on the positions of 
a finite collection of vertices $\{v_1, \ldots , v_\ell \} \subseteq \V$, we will be writing
$\mathcal{P}_{\{v_1, \ldots , v_\ell \},U}(N;\alpha, \nu)$
to denote the Poisson model consisting of the points 
of the vertices $\{v_1, \ldots , v_\ell \}$ together with 
the points of the Poisson process on $\D \setminus U$ 
with $N-\ell$ points in expectation. In other words, the set 
$X$ will consist of the outcomes of the random variables 
$\{v_1, \ldots , v_\ell \}$.

\subsection{Geometric properties of $\D$}
Recall that for any two points/vertices $u, v$ on $\D$, their \emph{relative angle} $\theta_{u,v}$ is defined as 
$\min \{ \hat{uOv}, 2 \pi - \hat{uOv} \}$. Recall also that $\theta_{u,v}\leq \pi$. 
We state and prove a simple geometric fact, which we will use several times in the following sections.
\begin{claim} \label{fact:triangle} Consider three vertices $z,y$ and $w$, on $\D$ (in the hyperbolic plane with curvature $-1$), 
such that $d_H (z,w) <R$ and $w$ is at the anticlockwise 
direction of $z$ whereas $y$ is between $z$ and $w$. 
If $t_y > t_w$, then $d_H (y,z) < R$. 
\end{claim}
\begin{proof}
This is the case as the point $y'$ of type equal to that of $y$ with $\theta_{y'w} = 0$ is still at distance less than $R$ 
from $z$. If we move this clockwise towards $z$, the distance will remain smaller than $R$, as $w$ will be at the
anticlockwise side of $y'$. 
\end{proof}


The following lemma provides a useful (almost) characterization of the fact that two vertices are within hyperbolic distance $R$, given their types.
The lemma reduces a statement about hyperbolic distances to a statement about the relative angle between two points.
Its proof can be found in \cite{Fount13+} and \cite{bode_fountoulakis_mueller}.
For two points $p,v$, let 
$$\hat{\theta}_{p,v}:=2(1+\eps ) e^{\frac{t_p+t_v- R}{2}}
= 2(1+\eps ) \frac{\nu}{N} e^{\frac{t_p+t_v}{2}} \text{, and}$$
$$\check{\theta}_{p,v}:=2(1-\eps ) e^{\frac{t_p+t_v- R}{2}}
=  2(1-\eps ) \frac{\nu}{N} e^{\frac{t_p+t_v}{2}}\text{.}$$
For $c_0 = c_0(\eps)$, that depends on $\eps$ as in the following lemma, we call the set 
$$T_{\eps}^+ (v):=\left\{p \in D_R  \ : \ t_p+t_v -R < -c_0, \ \theta_{p,v} \leq \hat{\theta}_{p,v}
\right\}$$  the  \emph{outer tube} of $v$. 
Similarly, we call the set 
$$T_{\eps}^- (v) := \left\{p \in D_R  \ : \ t_p+t_v -R < -c_0, \ \theta_{p,v} \leq \check{\theta}_{p,v} 
\right\}$$
the \emph{inner tube} of $v$. 

\begin{figure}[h]
\centering
\includegraphics[scale=0.2]{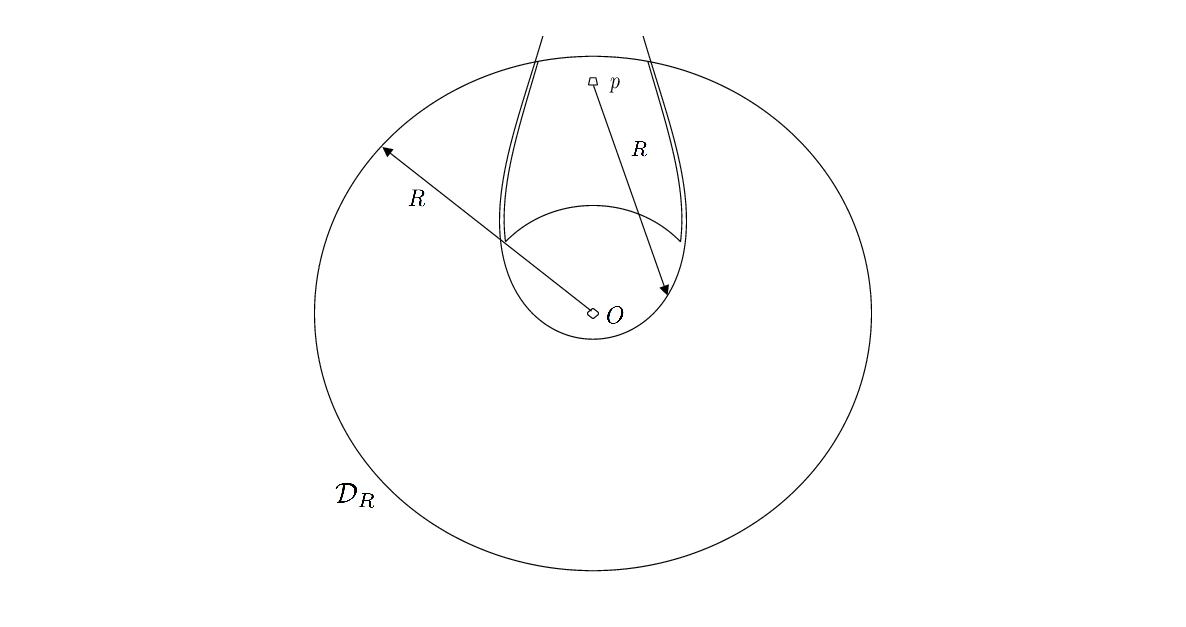}
\includegraphics[scale=0.2]{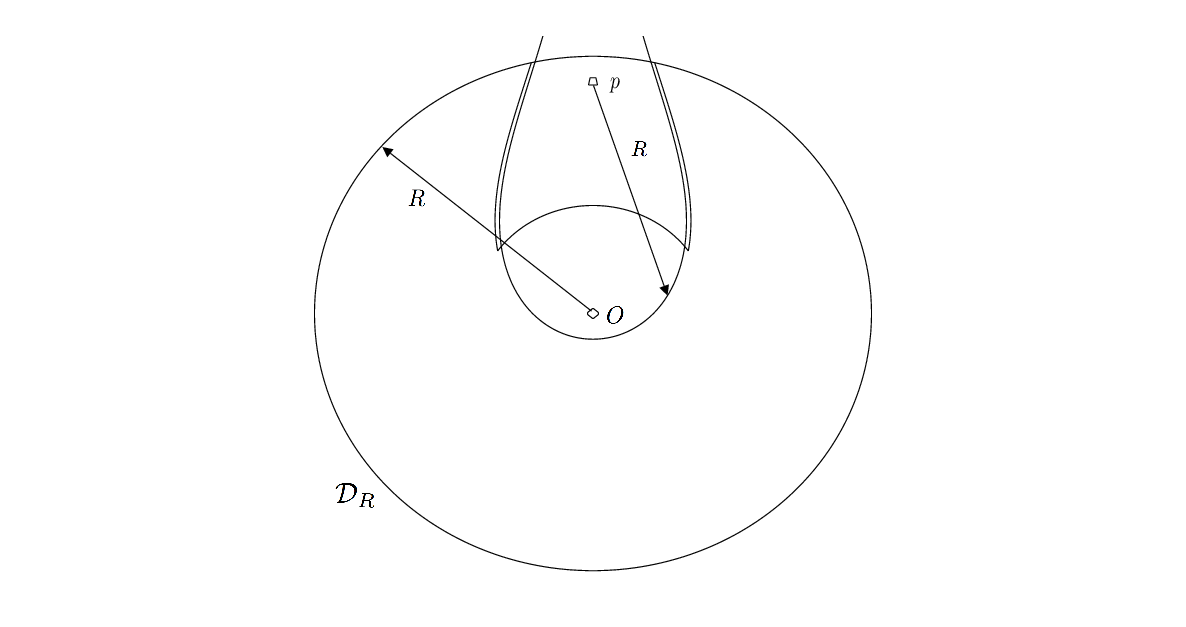}
 \caption{The inner and the outer tube around point $p$.}\label{fig:inn-out-tube}
 \end{figure}

\begin{lemma}\label{lem:relAngle}
For any $\eps>0$ there exists an $N_0>0$ and a $c_0>0$ such that for any $N>N_0$ and $u,v\in\D$ with $t_u+t_v < R-c_0$
the following hold.
\begin{itemize}
			\item If $u \in T_{\eps}^-(v)$, then $d_H (u,v) < R$.
			\item If $u \not \in T_{\eps}^+(v)$, then $d_H (u,v) > R$.
		\end{itemize}
\end{lemma}


\subsection{Properties of $\Gnan$}

We state some general results about the graphs, the proofs of which can be found in \cite{bode_fountoulakis_mueller}.

\begin{lemma}\label{lem:dist}
Let $\bar\rho(t)$ be the probability density function of the types.
For any $\eps \in (0,1)$, uniformly for $0\leq t<(1-\eps) R$ as $N\to \infty$ we have
\begin{equation}
 \bar\rho(t)=\rho(R-t)=(1+o(1))\alpha e^{-\alpha t}.
\end{equation}
\end{lemma}
The following fact is an immediate consequence of the above. 
\begin{corollary}\label{lem:innerEmpty}
	Let $\omega:\mathbb{N}\rightarrow\mathbb{N}$ be an increasing function such that $\omega (N) \rightarrow\infty$ as $N\rightarrow\infty$.
	The expected number of vertices of type at least $R/(2\alpha) + \omega (N)$ in $\Gnan$ is $o(1)$.
	Hence, with probability $1-o(1)$ all vertices in $V_N$ have type at most $\frac{1}{2\alpha}R+\omega(N)$.
\end{corollary}

\section{Proof of Theorem~\ref{thm:ultrasmall}: upper bound}
In this section we assume $1/2<\alpha<1$.
Recall that the definition of the type of a vertex as well as that of its radius use distances on the hyperbolic plane. 
\begin{definition}
	For $G\in\Pnan$ or $G\in \mathcal{G}(N;\alpha,\nu)$, let $\core(G)=\{v \in V(G) \ : \ t_v\geq R/2\}$ be the \textit{core} of $G$.
	Furthermore, we let $\xcore (G) = \{ v \in V(G) \ : \ t_v\geq R/2 - \log \log R \}$ be the \textit{extended core} of $G$. Finally, 
	we call the set $\per (G) := \{ v \in V(G) \ : \ t_v \leq \log\log R\}$ the set of \textit{peripheral vertices} of $G$. 
\end{definition}
Note that for every pair of vertices $u,v\in \core(G)$, by the triangle inequality the distance between $u$ and $v$ is at most $R$, so $uv\in E(G)$. In other words, the subgraph that is induced by the vertices in $\core (G)$ is complete. 

\begin{lemma}\label{lem:almostCore}
	Let $\omega(N)$ be such that $\omega(N)\rightarrow\infty$ as $N\rightarrow\infty$ but $\omega(N)=o(R)$. 
Let $x$ be a vertex such that $t_x < \log \log R$ and $U \subset \D$ an open subset of $\D$ which does not contain any points of 
type at least $\log \log R$ and has $\mu_\alpha (U)= o (\mu_\alpha (\D))$. 
Let $G \in \mathcal{P}_{\{x\},U} (N;\alpha, \nu)$.
	A.a.s. there is a vertex $u\in\core(G)$ such that $uv\in E(G)$ for every vertex $v$ with $t_v\geq \frac{2\alpha-1}{2\alpha}R+\omega(N)$.
\end{lemma}
\begin{proof}
	By the triangle inequality, any such vertex $v$ is adjacent to any vertex of radius at most $R(2\alpha-1)/(2\alpha)+\omega(N)$, so it is sufficient to
	show that a.a.s. the disc $D_r$ of radius $r:=\frac{2\alpha-1}{2\alpha}R+\omega(N)$ is non-empty.
	Note that $r<R/2$, for any $N$ large enough, as $\omega(N)=o(R)$ and $\alpha <1$, so any vertex in $D_r$ belongs to the core.
	Let $N_r$ be the number of vertices in $D_r$.
	
	Note first that $\frac{2\alpha -1}{2\alpha} - 1 = -\frac{1}{2\alpha}$. Thus $r-R=-\frac{R}{2\alpha} + \omega (N)$, whereby 
	$\alpha (r-R) = -R/2 + \alpha \omega (N)$. 
	As $D_r \cap U = \emptyset$, these identities imply that
	\begin{equation*}
	\begin{split}
	\Ex{N_r}&= (N-1) \frac{\mu_\alpha (D_r)}{\mu_\alpha (\D ) - \mu_\alpha (U)}
=(N-1)\frac{\cosh(\alpha r)-1}{\cosh(\alpha R)(1-o(1))} \\ &\sim N e^{\alpha (r-R)} 
= N e^{-R/2 + \omega (N)} \stackrel{N=\nu e^{R/2}}{=} \nu e^{\alpha \omega (N)}.
	\end{split}
	\end{equation*}
	Using this and Fact~\ref{fct:empty} we get
	\begin{align*}	
		\P(N_r\neq0)&=
			1-e^{-(1+o(1))\nu e^{\alpha\omega(N)}}=1-o(1).
	\end{align*}
\end{proof}
In fact, the only component we consider is the one containing the vertices in the core.  
We show that most pairs of vertices that are connected have a short path into the core.
These paths naturally give short paths connecting all the vertices in the component.
We are interested in the following paths in which the type of the vertices increases exponentially along the path.

\begin{definition}
For $\delta>0$, we call a path $P=v_1,v_2,\dots,v_m$ in $G$ a $\delta$-\emph{exploding path}
if $v_m\in \core(G)$ and $t_{v_{i+1}}\geq (1+\delta)t_{v_i}$ for $1\leq i\leq m-2$.
\end{definition}

Note that an exploding path must have length $O(\log_{1+\delta} R)$. 
Since $R = O(\log N)$ too, it turns out that such a path is ultra-short, that is, it has length $O(\log \log N)$. 

Not every vertex in the giant component has an exploding path into the core.
However, the vertices that do not have such a path are more likely to have a very low type.
In particular, we prove that any vertex of type at least $\log\log R$ has an exploding path into the core with probability $1-o(1)$. 
We actually show this lemma for the Poisson model. The result does transfer to $\Gnan$, due to its monotonicity, but we are going 
to use it later in this form. 
\begin{lemma}\label{lem:explode}
Let $\delta=2\frac{1-\alpha}{2\alpha-1}$ and $\zeta < \delta$ be a positive real number. 
Assume that $v$ and $x$ are vertices such that $t_v \geq \log \log R\geq t_x$ and $U \subset \D$ an open subset 
which does not contain any points of type at least $t_v$ so that $U$ is contained in a sector of $\D$ that spans a $o(1)$ angle. 
With probability (in the space $\mathcal{P}_{\{v,x\},U}(N;\alpha,\nu)$) 
$1-e^{-\Theta ( \log^{(\alpha - \frac{1}{2}) \zeta} R)}$, there is a $(\delta -\zeta)$-exploding path starting at $v$.
\end{lemma}
\begin{proof}
	Take any $\eps<\frac{1}{4}$ and assume that $N>N_0$, where $N_0$ is as in Lemma~\ref{lem:relAngle}.
	
	By Lemma~\ref{lem:almostCore}, if $v$ satisfies $t_v\geq \frac{2\alpha -1}{2\alpha}R + \omega (N)$,
	then a.a.s. there is a vertex $u\in G$ with $t_{u}\geq R/2$ and $vu\in E(G)$.
	In other words, if $t_v\geq \frac{2\alpha -1}{2\alpha}R + \omega (N)$, then we are done.
	

    Assume now that $t_v < \frac{2\alpha -1}{2\alpha}R + \omega (N)$. As $1+\delta = \frac{1}{2\alpha -1}$, it follows that 
    $(1+\delta) t_v< \frac{1}{2\alpha} R + \frac{\omega (N)}{2\alpha -1}$. Note that by Corollary~\ref{lem:innerEmpty},
    it suffices to consider only points of type no larger than $\frac{1}{2\alpha} R + \frac{\omega (N)}{2\alpha -1}$.
    
    Let $v_1=v$. We will construct inductively a series of (random) sets $T_i \subset \D$, for $i\geq 2$, in each of which we find a
 vertex $v_i$, which will be the $i$th vertex in the exploding path.

		For two points $p,p'$, let 
		$\vartheta_{p,p'}= \theta_{p,p'}$ if $p'$ is in the anti-clockwise direction from $p$, but $\vartheta_{p,p'}= -\theta_{p,p'}$, otherwise.
    
    Assume that we have exposed $v_i$. For any point $p \in \D$ we let
    
    $$ \hat{T}_\eps^- (p) := \left\{ p' \in \D \ : \ |t_{p'} - (1+\delta )t_{p}| < \zeta t_p, \ \frac{\eps \nu}{N} e^{\frac{t_p + t_{p'}}{2}} \leq \vartheta_{p',p} \leq \frac{2(1-\eps) \nu}{N} e^{\frac{t_{p'} + t_p}{2}}\right\}.$$
    We take $T_i := \hat{T}_{\eps}^- (v_i)$.
	Let $A$ be the set of vertices that are located in $\hat{T}_{\eps}^- (v_i)$.
	Note that, as the angle covered by $U$ is $o(1)$, we have that $\mu_\alpha (U) = o (\mu_\alpha (\D))$.
	Hence, the area of a set in $\D \setminus U$ is within a $1-o(1)$ factor from the area in $\D$ (both on the hyperbolic 
plane of curvature $-\alpha^2$). 
	
	So, for any $\eps \in (0,1/4)$ and for $N$ large enough we have
	
	\begin{align*}
		\Ex{|A|}&\geq 2(1-\frac32\eps ) \frac{N-2}{2\pi}\int_{(1+\delta -\zeta)t_{v_i}}^{(1+\delta + \zeta)t_{v_i}}  e^{\frac{1}{2}(t_{v_i}+t-R)}
		(1-o(1))e^{-\alpha t}dt\\
			&\geq 2(1-\frac32\eps )(1-o(1)) \frac{N}{2\pi}\frac{\nu}{N}e^{\frac{t_{v_i}}{2}}
				\int_{(1+\delta -\zeta)t_{v_i}}^{(1+\delta + \zeta)t_{v_i}}e^{(\frac12-\alpha)t}dt\\
			&\stackrel{\eps<1/4}{\geq}\frac{\nu}{2\pi}e^{\frac{1}{2}t_{v_i}}\frac{1}{2\alpha-1}
				\left(e^{(\frac12-\alpha)(1+\delta -\zeta )t_{v_i}}-e^{(\frac12-\alpha) (1+\delta + \zeta)t_{v_i}}\right).
	\end{align*}
	But $(1+\delta + \zeta)t_{v_i} - (1+\delta)t_{v_i} + \zeta t_{v_i}> 2\zeta t_{v_i} \rightarrow \infty$, whereby the above becomes:
	\begin{align*}
	\Ex{|A|}&\geq \frac{\nu}{2\pi}\frac{1}{2\alpha-1}e^{\frac{1}{2}t_{v_i}- (\alpha -\frac12 )(1+\delta -\zeta)t_{v_i} }(1-o(1)). 
	\end{align*}
Furthermore, $(\alpha - \frac12)(1+\delta)=\frac{2\alpha-1}{2}\frac{1}{2\alpha-1}=\frac{1}{2}$ and finally, we have  
$$\Ex{|A|}\geq \frac{\nu}{2\pi}\frac{1}{2\alpha-1}e^{(\alpha - \frac12)\zeta t_{v_i}}(1-o(1)) \stackrel{2\alpha -1 < 1}{\geq}
\frac{\nu}{2\pi} e^{(\alpha - \frac12)\zeta t_{v_i}}, $$
for $N$ large enough. 
Hence, by Fact~\ref{fct:empty} we have 
	\begin{align*}
		\P(|A|>0) &=1-\P(|A|=0)\\
		&\geq1-\exp\left(-\frac{\nu}{2\pi} e^{(\alpha - \frac12) \zeta t_{v_i}}\right).
	\end{align*}
	As $t_{v_i}\geq\log\log R$, we have $\P(|A|=0) \leq\exp\left(-\frac{\nu}{\pi}(\log R)^{(\alpha - \frac12) \zeta}\right)$.
	If $|A|>0$, then there are vertices that are located inside $T_i$ and we let $v_{i+1}$ be one of them -- the choice is arbitrary. 
    The following claim guarantees that $T_{i+1}=\hat{T}_\eps^- (v_{i+1})$ is disjoint from $T_i$ and when we repeat the 
    argument there is no danger to expose again area which we have already exposed. 
    \begin{claim} \label{clm:reg_disjoint} 
For all $N$ large enough and for all $i\geq 1$ the following holds.
For all $p \in \hat{T}_{\eps}^- (v_i)$ we have 
$T_\eps^+ (v_i) \cap \hat{T}_{\eps}^-(p) = \emptyset$. \end{claim}
 \begin{proof}[Proof of Claim~\ref{clm:reg_disjoint}]
 Consider a point $p \in \hat{T}_{\eps}^- (v_i)$ and let $p' \in \hat{T}_\eps^- (p)$. 
We will show that 
$$\vartheta_{p',v_i} \gg 2(1+\eps ) \frac{\nu}{N} e^{\frac{t_v + t_{p'}}{2}}. $$
We write $\vartheta_{p',v_i} = \vartheta_{p',p} + \vartheta_{p,v_i}$. 
Since $p' \in \hat{T}_\eps^- (p)$ and $p \in \hat{T}_{\eps}^- (v_i)$ we have 
$$ \vartheta_{p',p} \geq \eps \frac{\nu}{N} e^{\frac{t_{p'} + t_p}{2}} \ \mbox{and} \ \vartheta_{p, v_i} 
\geq \eps \frac{\nu}{N} e^{\frac{t_p + t_{v_i}}{2}}.$$
Hence 
\begin{equation*}
\begin{split} 
\vartheta_{p',p} + \vartheta_{p,v_i} & \geq  \eps \frac{\nu}{N} \left( e^{\frac{t_{p'} + t_p}{2}}  + e^{\frac{t_p + t_{v_i}}{2}} \right) \\
& =\eps \frac{\nu}{N} e^{\frac{t_{p'} + t_{v_i}}{2}} \left( e^{\frac{t_p - t_{v_i}}{2}} + e^{\frac{t_p - t_{p'}}{2}}\right) > 
\eps \frac{\nu}{N} e^{\frac{t_{p'} + t_{v_i}}{2}}  e^{\frac{t_p - t_{v_i}}{2}} \\
&\geq \eps \frac{\nu}{N} e^{\frac{t_{p'} + t_{v_i}}{2}} e^{(\delta -\zeta)t_{v_i}} \stackrel{(\delta -\zeta)t_{v_i} \rightarrow \infty}{\gg}
2(1+\eps ) \frac{\nu}{N} e^{\frac{t_{p'} + t_{v_i}}{2}}.
\end{split}
\end{equation*}
In fact, $(\delta -\zeta)t_{v_i} \geq (\delta -\zeta)\log \log R$, and therefore the inequality holds uniformly for all $N$ that are large enough. 
 \end{proof}

	If we start at type at least $\log \log R$, it takes $O(\log R )$ steps to reach type $\frac{2\alpha -1}{2 \alpha}R+\omega (N)$;
	at that point we can complete the exploding path using the vertex whose existence is guaranteed by Lemma~\ref{lem:almostCore}.
	Thus for any given vertex $v$ with $t_v>\log\log R$ we have
	\begin{align*}
		\P(\exists\text{ sequence of vertices $v_2,\ldots$})&=\left(1-\exp\left(-\frac{\nu}{\pi}(\log R)^{(\alpha - \frac12) \zeta}\right)\right)^{O(\log R)}\\
			&=1-O(\log R )\exp\left(-\frac{\nu}{\pi}(\log R)^{(\alpha - \frac12) \zeta}\right)\\
			&=1- \exp\left(-\Theta \left( \log^{(\alpha - \frac12 ) \zeta} R \right) \right),
	\end{align*}
	as $xe^{-ax^b}=o(1)$ for $0<a,b$ and $x \rightarrow \infty$.
	
\end{proof}

\begin{remark}\label{rem:alwaysExplode}
In fact, if the type of $v$ is $O(1)$, that is, $v$ is a typical vertex, then the probability that there is a $(\delta - \zeta)$-exploding path 
starting at $v$ is bounded away from $0$. With slightly more work, one can show that two vertices $u$ and $v$ have both an exploding 
path with probability that is asymptotically bounded away from $0$. Thus, $d_G (u,v)<\infty$ with probability that is asymptotically bounded
away from $0$. Alternatively, this follows from the main theorem in~\cite{bode_fountoulakis_mueller}, according to which $\Gnan$ has 
giant component a.a.s. if $1/2 < \alpha < 1$. 
\end{remark}

Given two vertices $u$ and $v$ that do not belong to $\core(G)$, if there are $(\delta - \zeta)$-exploding paths starting at $u$ and $v$, respectively, then $d_(G) (u,v) < \left(\frac{2}{\log (1+\delta -\zeta)} \log R\right) +1$. 
In particular, with $\delta = 2\frac{1-\alpha}{2\alpha-1}$ (as in the above lemma), we have $1+\delta = \frac{1}{2\alpha -1}$ and thereby $\frac{2}{\log (1+\delta )} = 2 \tau$. 
Hence, to deduce the upper bound in Theorem~\ref{thm:ultrasmall}, it suffices to show that the probability that $u$ and $v$ belong to the same connected component, but either $u$ or $v$ do not belong to a $(\delta -\zeta)$-exploding path 
is $o(1)$. We are now ready to proceed with the details. 
\begin{proof}[Proof of Theorem~\ref{thm:ultrasmall}: upper bound] 
Let $u,v$ be two arbitrary distinct vertices in $\V$. We will show that the event $d_G (u,v) < \infty$ but $d_G(u,v) \geq (2\tau + \zeta) \log R$ occurs with 
probability $o(1)$. Note that this is in the $\Gnan$ space. 
We denote this event by $\mathcal{E}_{N}(\tau,\zeta)$. 
Also, for some fixed $\eps \in (0,1)$, let $\mathcal{A}_N$ denote the event that the relative angle between $u$ and $v$ is 
greater than $\nu \frac{2\zeta_\eps \log R}{N}$,  where $\zeta_\eps := \zeta^2 (1-\eps)$.  
Since the angles of the points $u$ and $v$ are independent and uniformly distributed, the probability of $\overline{\mathcal{A}_N}$ is $o(1)$ and therefore it suffices to prove that 
$\Pro {\mathcal{E}_N (\tau,\zeta) \cap \mathcal{A}_N} = o(1)$.

If $\mathcal{E}_{N}(\tau,\zeta)$ is realised, then there must be a \emph{minimal} path between vertices $u$ and $v$. 
In this context, a minimal path is meant to be an induced path. Let $P_{\min}$ denote such a path. 
Assume, in addition, that $\mathcal{A}_N$ is simultaneously realised, that is, $\theta_{u,v} > \nu \frac{2\zeta_\eps \log R}{N}$. 
With this assumption, 
let $P_{\min}(u)$ denote the sub-path of $P_{\min}$ starting at $u$ and ending at the first vertex whose relative angle with 
$u$ exceeds $\nu \frac{\zeta_\eps \log R}{N}$. 
Similarly, let $P_{\min}(v)$ denote the sub-path of $P_{\min}$ starting at $v$ and ending at the first vertex whose relative angle
with $v$ exceeds $\nu \frac{\zeta_\eps \log R}{N}$. Clearly, as $\mathcal{A}_N$ is realized, the two paths may overlap, but they 
have at most one edge in common.  

Assume without loss of generality that $v$ is at angle $\theta_{u,v}\leq \pi$ in the anti-clockwise direction from $u$.
Consider the sectors consisting of points of relative angle at most $\nu \frac{\zeta_\eps \log R}{N}$ from a point $x$:
$$S_h^+(x):= \left\{ p \in  \D \ : \ 
t_p > \log \log R, \ 0<\vartheta_{x,p} < \nu \frac{\zeta_\eps \log R}{N} \right\}$$ 
and 
$$ S_h^- (x) := \left\{ p \in  \D \ : \ 
t_p > \log \log R, \  -\nu \frac{\zeta_\eps \log R}{N} < \vartheta_{x,p} <0 \right\}\text{.}$$
 
There are two cases:
\begin{enumerate}
\item[1.] either each  one of $S_h^+(u), S_h^-(u), S_h^+(v), S_h^-(v)$ contains a vertex that is the starting vertex of a $(\delta- \zeta^2)$-exploding path,
\item[2.] or at least one of them is either empty or none of its vertices is the endpoint of a $(\delta-\zeta^2)$-exploding path. 
\end{enumerate}
Let $\mathcal{S}$ denote the former and let $\overline{\mathcal{S}}$ denote the latter. 

	\begin{figure}[h]
\centering
\includegraphics[width=0.5\columnwidth]{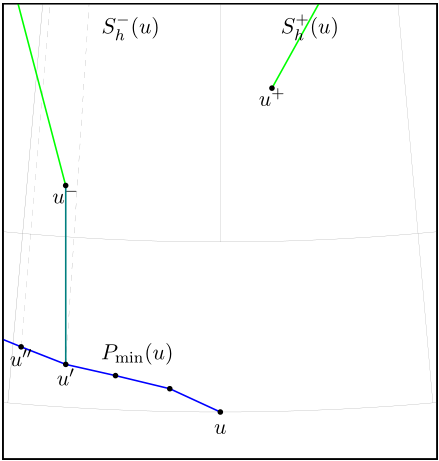}
 	\caption{Creating a short path into the core under $\mathcal{S}$.}\label{fig:short}
 \end{figure}

We will show that $\Pr (\overline{\mathcal{S}}) = o(1)$. 
First consider, without loss of generality, the set $S_h^+ (u)$. 
The probability that this set is empty is $o(1)$.  
Indeed, let $N_{S_h^+(u)}$ be the number of vertices that appear into this sector. 
Then
$$ \Ex {N_{S_h^+(u)}} = N \frac{\cosh (\alpha (R-\log \log R))-1}{\cosh(\alpha R)-1}
~\frac{1}{2\pi}~\nu\frac{\zeta_\eps \log R}{N} \approx \log^{1-\alpha} R \rightarrow \infty. 
$$
The distribution of $N_{S_h^+(u)}$ is binomial and the application of a standard Chernoff bound implies that 
$\Pro { N_{S_h^+ (u)}=0} = o(1)$. 

If $S_h^+(u)$ is not empty and none of its vertices is the beginning of a $(\delta -\zeta^2)$- exploding path, then the vertex with 
lowest type in $S_h^+(u)$ does not have a $(\delta-\zeta^2)$-exploding path starting at it as well. We call this vertex the \emph{first} 
vertex in $S_h^+(u)$. 
\begin{claim} \label{clm:first_vertex}
The  probability that the first vertex in $S_h^+(u)$ does not have a $(\delta-\zeta^2)$-exploding path starting at it is o(1). 
\end{claim} 
\begin{proof}[Proof of Claim~\ref{clm:first_vertex}]  
Conditional on having at least one vertex in $S_h^+ (u)$, let $u'$ be the first vertex (with probability 1 there will be exactly one 
such vertex) 
which we expose and assume that the area in $S_h^+(u)$ that consists of points with type greater than $t_{u'}$ has not been exposed. 
Let us switch temporarily to $\mathcal{P}_{X, U} (N;\alpha,\nu)$, where $X=\{u,u'\}$ and $U$ the subset of $S_h^+(u)$ below $u'$.
Then taking $\zeta$ to be $\zeta^2$ in Lemma~\ref{lem:explode}, we deduce that  
there is a $(\delta- \zeta^2)$-exploding path starting at $u'$ with probability $1-o(1)$ uniformly over 
$t_{u'} \geq \log \log R$. This lemma can be applied as the area above $u'$ has not been exposed in the corresponding Poisson process
and the proof of Lemma~\ref{lem:explode} deals only only with that area.
The result transfers to $\Gnan$ (conditional on $U$ being empty and on the realisations of $u$ and $u'$), 
through Lemma~\ref{lem:poisson_cond}, due to the fact that this property is non-decreasing.  
\end{proof}
Then, since the probability that $S_h^+(u)$ is empty is $o(1)$, the union bound implies that $\Pro {\overline{\mathcal{S}}} = o(1)$. 

We will show that $\Pro {\mathcal{E}_N (\tau,\zeta) \cap \mathcal{A}_N \cap \mathcal{S}} = 0$. 
Observe that any vertex which belongs to $S_h^+(u) \cup S_h^-(u)$ (or to $S_h^+(v) \cup S_h^-(v)$, respectively) will be adjacent to a
vertex in $P_{\min}(u)$ ($P_{\min} (v)$, resp.). Indeed, if $P_{\min} (u)$ contains a vertex in $S_h^+(u) \cup S_h^-(u)$, then 
this must be adjacent to any other vertex in $S_h^+(u) \cup S_h^-(u)$. This is the case as $S_h^+(u) \cup S_h^-(u) \subseteq 
T_\eps^- (u')$ for any $u' \in S_h^+(u) \cup S_h^-(u)$, provided that $\zeta <1$. To see this, note that any two points 
in $S_h^+(u) \cup S_h^-(u)$ have relative angle at most $2 \zeta_{\eps} \frac{\nu}{N}$. However, for any point in 
$S_h^+(u) \cup S_h^-(u)$, its inner tube consists of all points of relative angle at most $2(1-\eps )\frac{\nu e^{\log \log R}}{N}$ from 
it. Thus, if $\zeta_{\eps} < 1-\eps$ (that is, $\zeta < 1$), then the containment follows. 
In this case, some vertex of $P_{\min}(u)$ will be connected to the first vertex in $S_h^+(u) \cup S_h^-(u)$. 

Suppose now that all vertices of $P_{\min}(u)$ do not belong to $S_h^+(u) \cup S_h^-(u)$. Let $u^+, u^-$ be vertices in 
$S_h^+(u)$ and $S_h^-(u)$ respectively, which are the starting vertices of $(\delta-\zeta^2)$-exploding paths $P_{u^+}$ and $P_{u^-}$. 
There are two \emph{consecutive} vertices in $P_{\min}  (u)$ say $u',u''$ such that 
either $\vartheta_{u'',u^+}> 0 > \vartheta_{u',u^+}$ or $\vartheta_{u'',u^-}> 0 > \vartheta_{u',u^-}$. Thus, 
either $u^+$ or $u^-$ has type larger than $t_{u'}$ or $t_{u''}$ and therefore by Fact~\ref{fact:triangle} either $u^+$ or $u^-$ is adjacent to 
at least one of $u'$ or $u''$. 
The length of any exploding path is at most $\log R/ \log (1+\delta - \zeta^2)$. Thus, 
$|P_{u^+}|, |P_{u^-}| \leq  \log R/ \log (1+\delta - \zeta^2)$. 
 The following bounds the length of $P_{\min}(u), P_{\min} (v)$:
\begin{claim} \label{clm:minpathlength}
Both $P_{\min} (u)$ and $P_{\min} (v)$ have length at most $\zeta^2 \log R$. 
\end{claim}
\begin{proof}[Proof of Claim~\ref{clm:minpathlength}]
Consider $P_{\min}(u)$ (the proof for $P_{\min}(v)$ is identical). Since $P_{\min}(u)$ is part of a minimal path, it follows that if 
we take the set of vertices of $P_{\min}$ that are at even distance from $u$, then there cannot be an edge between any two of 
them, for this would contradict the minimality of $P_{\min}$.  Let $P^{e}_{\min}(u)$ be this set of vertices. 
For any vertex $u'\in P^{e}_{\min}(u)$ consider the sector 
$T(u'):= \{ p \in \D \ : \ \theta_{u',p} < (1-\eps) {\nu \over N} \}$.  There cannot be distinct $u',u'' \in P^{e}_{\min}(u)$ such that 
$T(u') \cap T(u'') \not = \emptyset$. If this were the case, then their relative angle would be at most $2(1-\eps ){\nu \over N}$
and by Lemma~\ref{lem:relAngle} they would be adjacent.  
But there are at most $\nu \frac{\zeta_\eps \log R}{N}/ \left( 2 (1-\eps) {\nu \over N} \right)= \frac{\zeta^2}{2} \log R$ such sectors 
inside the sector of angle $\nu \frac{\zeta_\eps \log R}{N}$ in the anti-clockwise direction from $u$.  
Thus $|P^{e}_{\min}(u)| \leq \frac{\zeta^2}{2} \log R$, whereby the length of $P_{\min}$ is at most 
$\zeta^2 \log R$.
 \end{proof}
Thus 
\begin{equation*}
\begin{split}
d_G (u,v) &\leq |P_{\min}(u)| + |P_{u^+}|  +1 + |P_{u^-}| + |P_{\min} (v)|\\ 
& \leq 2\left( \frac{1}{\log (1+\delta - \zeta^2)} + \zeta^2 +o(1) \right) \log R 
\end{split}
\end{equation*}
Hence, there exists a $\zeta$ such that for all $N$ large enough 
$ \frac{1}{\log (1+\delta - \zeta^2)} + \zeta^2 +o(1) < \tau + \zeta$. This implies that $\mathcal{E}_N(\tau,\zeta)$ is not 
realised. 
\end{proof}

\begin{remark}
If we replace the angles that determine the domains $S_h^+$ and $S_h^-$ by a quantity that is proportional to 
$R^{\frac{1}{1-\alpha}}/N$ and the lower bound on the type by $\frac{1}{2(1-\alpha )}\log R$, then 
the probabilities that appear above become $o(N^{-2})$. Thus, the analogous of the above bound on $d_G (u,v)$ holds 
for all pairs of vertices, and implies that the diameter is proportional to $R^{\frac{1}{1-\alpha}}$ a.a.s. This upper bound 
is worse than the one obtained in~\cite{FriedKroh15}. 
\end{remark}

\section{Proof of Theorem~\ref{thm:ultrasmall}: lower bound}

For given vertices $u,v\in \V$, let $\mathcal{L}_{\zeta,N}(u,v)$ be the event that $d_G(u,v)<(2\tau-\zeta)\log R=:L$, for some $\zeta>0$.
Assume that $u$ and $v$ are peripheral vertices, that is, $t_u,t_v<\log\log R$ - by Lemma~\ref{lem:dist} this event occurs with probability $1-o(1)$.
Let $\mathcal{T}_{u,v}$ denote this event.
By Lemma~\ref{lem:relAngle}, for any $T\leq R/2-2\log\log R$, if $u$ and $v$ are connected through a path of length at most $L$
where the intermediate vertices have type at most $T$, then
\[\theta_{u,v}\leq4\nu\frac{e^T}{N}L \leq4\nu\frac{e^{R/2}}{N}\frac{L}{\log^2R}=4\frac{L}{\log^2R}.\]
Conditional on $\mathcal{T}_{u,v}$, the probability of this event is $O(L/\log^2R)=o(1)$.
Now, if there is a path of length at most $L$ that joins $u$ to $v$ that contains an intermediate vertex of type at least $R/2-2\log\log R$,
then there must be a path of length at most $L/2$ either from $u$ or from $v$ to this vertex.
Denote by $d_G (u,\xcore)$ the graph distance of the vertex $u$ to the extended core, that is to the set of vertices of type at least $R/2 - \log \log R$. 
The following lemma proves that almost all vertices are, in some sense, far away from vertices this type, immediately proving the lower bound.

\begin{lemma}\label{lem:distCore}
	Assume that $t_u\leq \log\log R$. 
	For $\zeta >0$, we have
	\[\Pr(d_G (u,\xcore)\leq (\tau-\zeta^{1/2})\log R)=o(1)\text{.}\]
\end{lemma}
We appeal to Lemma~\ref{lem:poisson_cond} for the event $\{ d_G(u,\xcore)\leq (\tau-\zeta^{1/2})\log R \}$. Clearly, this is a non-decreasing 
event in the sense that is used in that lemma. So, it suffices to prove Lemma~\ref{lem:distCore} in the $\mathcal{P}_{\{u\},\emptyset} (N;\alpha,\nu)$ 
space. 

To prove this statement, we keep track of the highest type in the neighbourhood of the vertex $u$.
Let $N^{(0)}(u)=\{u\}$, $\theta_r^{(0)}=\theta_\ell^{(0)}=0$.
For $i\geq0$, define $N^{(i)}(u)$ as the neighbours of vertices in $N^{(i-1)}(u)$ that are in clockwise direction of $u$ and have relative angle
greater than $\theta_\ell^{(i-1)}$ with $u$ or that are in anticlockwise direction of $u$ and have relative angle with $u$ 
greater than $\theta_r^{(i-1)}$.
Define $\theta_r^{(i)}$ as the maximum relative angle between $u$ and any vertex in $N^{(i)}(u)$ that is in anticlockwise direction of $u$,
setting it to $\theta_r^{(i-1)}$ if there is no such vertex.
Similarly, define $\theta_\ell^{(i)}$ as the maximum relative angle between $u$ and any vertex in $N^{(i)}(u)$ that is in clockwise direction of $u$,
setting it to $\theta_\ell^{(i-1)}$ if there is no such vertex.
This is the simultaneous breadth exploration process that will be defined in more detail in the next section.

Note that any vertex in $N^{(i)}(u)$ has graph distance $i$ to $u$, but not every vertex of distance $i$ is in $N^{(i)}(u)$.
However, we claim that the process cannot  leave a vertex that has type larger than the maximum type of any vertex in
$N_i(u):=\bigcup_{j=0}^i N^{(j)}(u)$ and is undiscovered within the sectors that have been exposed. 
For the sake of contradiction, assume that $v$ is a vertex whose  type is larger than the types of 
all vertices discovered in $N_i (u)$, but its angle with $u$ satisfies $\theta_{r}^{(k-1)}<\vartheta_{u,v}\leq \theta_r^{(k)}$, 
for some $1\leq k\leq i$. Then there are two vertices $v_{k-1} \in N^{(k-1)} (u)$ and $v_k \in N^{(k)}$ such that 
$v$ is between them; that is, $\vartheta_{v,v_{k-1}} < 0 \leq \vartheta_{v,v_k}$.  
Applying Fact~\ref{fact:triangle} with $v_{k-1},v,v_k$ playing the role of $z,y,w$ implies that $v$ is adjacent to $v_{k-1}$ and 
therefore should have been discovered and become a member of $N^{(k)}(u)$.


The above claim has also the following consequence. 
Denote by $t^{(i-1)}$ the maximum type of a vertex in $N_{i-1}(u)$.
As every vertex in $N^{(i-1)}(u)$ is further in the anticlockwise or in the clockwise direction, in terms of relative angle from $u$, than all the vertices in $N_{i-2}(u)$,
all vertices in $N^{(i)}(u)$ are either within (hyperbolic) distance $R$ and in the clockwise direction of the point $p_\ell^{(i-1)}$ of type
$t^{(i-1)}$ and of clockwise relative angle $\theta^{(i-1)}_\ell$ to $u$, or within (hyperbolic) distance $R$ and in the anticlockwise
direction
of the point $p_r^{(i-1)}$ of type $t^{(i-1)}$ and of clockwise relative angle $\theta^{(i-1)}_r$ to $u$.
Thus the highest type of a vertex in $N^{(i)}(u)$ is stochastically dominated from above by the highest type among all vertices
that have hyperbolic distance at most $R$ from a certain point of type $t^{(i-1)}$ (namely $p_r^{(i-1)}$ or $p_\ell^{(i-1)}$).
Due to this we can bound the distribution function of $t^{(i)}$ from below using Fact~\ref{fct:empty}.
Let $\hat{t}^{(i)}:=(1+\delta +\zeta)^i t_u$, for any integer $i\geq 0$.

\begin{claim}\label{clm:step}
	For $i\geq 1$, assuming that $\hat{t}^{(i-1)}<\frac{R/2-2\log\log R}{1+\delta+\zeta}$, we have
	\[\Pr(t^{(i)} < (1+\delta+\zeta)\hat{t}^{(i-1)} \ | \ t^{(i-1)} < \hat{t}^{(i-1)})\geq\exp\left(-\frac{2\nu}{(\alpha-1/2)\pi}e^{-(\alpha-1/2)\zeta \hat{t}^{(i-1)}}\right)\text{.}\]
\end{claim}
\begin{proof}
    By the assumption of the claim, if $t^{(i-1)} < \hat{t}^{(i-1)}$, then
    $t^{(i-1)}<(1/(1+\delta+\zeta))(R/2-2\log\log R)<(2\alpha-1)R/2$. 
	Lemma~\ref{lem:relAngle} works for types $t$ such that $t+t^{(i-1)}<R-c_0$ for a given constant $c_0$,
	so $t<R-(1/(1+\delta))R/2$ will do.
	Recall that $1/(1+\delta)=2\alpha-1$, so $t<R(3/2-\alpha)$ is sufficient.
	But $3/2-\alpha>1/(2\alpha)$, and so if we take  $\hat{t}=R/(2\alpha) +\omega(N)$, for some sufficiently slowly growing function $\omega(N)$, we are able to 
    use Lemma~\ref{lem:relAngle} for points of type at most $\hat{t}$.
	The first part of Corollary~\ref{lem:innerEmpty} implies that the expected number of vertices of type at least $\hat{t}$ 
    in $\mathcal{G}_{\{u\},\emptyset} (N;\alpha, \nu)$ is $o(1)$.

	As discussed above, the event where $t^{(i)}\leq (1+\delta+\zeta)\hat{t}^{(i-1)}$ has no smaller probability 
    than the event that a vertex
	of type $\hat{t}^{(i-1)}$ has no neighbour of type at least $\hat{t}^{(i)}$.
	Thus by Fact~\ref{fct:empty} and Lemma~\ref{lem:relAngle}, for $\eps>0$ small enough so that $(1+2\eps)\alpha<1$ we have
	\begin{align*}
		&\Pr\left(t^{(i)} < (1+\delta+\zeta)\hat{t}^{(i-1)}\ | \  t^{(i-1)} < \hat{t}^{(i-1)} \right)\\
			&\geq\exp\left(-N\int_{(1+\delta+\zeta)\hat{t}^{(i-1)}}^{\hat{t}}\frac{4(1+\eps)}{2\pi}e^{1/2(t+\hat{t}^{(i-1)}-R)}\alpha
				e^{-\alpha t}dt + o(1)\right)\\
			&\geq\exp\left(-\frac{2(1+2\eps)\alpha\nu}{\pi}e^{\frac{\hat{t}^{(i-1)}}{2}}\int_{(1+\delta+\zeta)\hat{t}^{(i-1)}}^{\infty}e^{(1/2-\alpha)t}dt\right)\\
			&\geq\exp\left(-\frac{2(1+2\eps)\alpha\nu}{\pi}e^{\frac{\hat{t}^{(i-1)}}{2}}
				\frac{1}{\alpha-1/2}e^{(1/2-\alpha)(1+\delta+\zeta)\hat{t}^{(i-1)}}\right)\\
			&\geq\exp\left(-\frac{2(1+2\eps)\alpha\nu}{\pi}e^{\frac{\hat{t}^{(i-1)}}{2}}
				\frac{1}{\alpha-1/2}e^{(-1/2+(1/2-\alpha)\zeta)\hat{t}^{(i-1)}}\right)\\
			&\geq\exp\left(-\frac{2\nu}{(\alpha-1/2)\pi}e^{-(\alpha-1/2)\zeta \hat{t}^{(i-1)}}\right)\text{,}
	\end{align*}
	as $(\alpha-1/2)(1+\delta)=1/2$.
\end{proof}

We repeatedly apply this bound to bound the distance from the core.
Assume that $t_u=\log \log R$.
Denote by $\mathcal{U}$ the event that if we explore as above the neighbours $u$ for every 
$i <(\tau-\zeta^{1/2})\log R$ we have $t^{(i)} < \hat{t}^{(i)}$.
\begin{claim}\label{clm:allGood} Assume that $t_u=\log \log R$. 
For $\zeta >0$ small enough (depending on $\alpha$), the event
	$\mathcal{U}$ has probability $1-o(1)$ and after the steps are completed the maximum type reached is less than $R/2-2\log\log R$, 
if $N$ is sufficiently large.
\end{claim}
\begin{proof}
On this event, after executing the $(\tau-\zeta^{1/2})\log R$ steps we have reached type less than
\begin{align*}
	    &(1+\delta+\zeta)^{(\tau-\zeta^{1/2})\log R}\log\log R=e^{\log(1+\delta+\zeta)(\tau-\zeta)\log R}\log\log R\\
		&\leq R^{(\log(1+\delta)+\zeta)(\tau-\zeta^{1/2})}\log\log R\\
		&= R^{(\tau^{-1}+\zeta)(\tau-\zeta^{1/2})}\log\log R\\
		&= R^{1-\tau^{-1}\zeta^{1/2}+\tau\zeta-\zeta^{3/2}}\log\log R = o( R/2-2\log\log R).
\end{align*}
Moreover, we are able to apply Claim~\ref{clm:step} repeatedly for this number of steps and deduce that $\mathcal{U}$ has probability 
\begin{align*}
	\Pr(\mathcal{U})&\geq\prod_{i=0}^{(\tau-\zeta^{1/2})\log R}\exp\left(-\frac{2\nu}{(\alpha-1/2)\pi}e^{-(\alpha-1/2)\zeta
			(1+\delta+\zeta)^i\log\log R}\right)\\
		&\geq\prod_{i=0}^{(\tau-\zeta^{1/2})\log R}\left(1-\frac{2\nu}{(\alpha-1/2)\pi}e^{-(\alpha-1/2)\zeta (1+\delta+\zeta)^i\log\log R}\right)\\
		&\geq1-\sum_{i=0}^{(\tau-\zeta^{1/2})\log R}\frac{2\nu}{(\alpha-1/2)\pi}e^{-(\alpha-1/2)\zeta (1+\delta+\zeta)^i\log\log R}\\
		&\geq1-\frac{4\nu}{(\alpha-1/2)\pi}e^{-(\alpha-1/2)\zeta\log\log R}=1-o(1)\text{.}
\end{align*}
\end{proof}

\begin{proof}[Proof of Lemma~\ref{lem:distCore}]
	Fact~\ref{fact:triangle} implies that increasing the type of a vertex will keep all edges intact, 
    so any path will stay a path if we increase the type of one of its vertices.
	Thus by a simple coupling argument we have that $\Pr(d(u,\xcore)\leq d|t_u)\leq\Pr(d(u,\xcore)\leq d|t_u')$ for $t_u\leq t_u'$.
	We can thus assume that $t_u=\log\log R$.
	By Claim~\ref{clm:allGood}, a.a.s. executing $(\tau-\zeta^{1/2})\log R$ steps yields maximum type that is less than $R/2-2\log\log R$, so
	\[\Pr(d(u,\xcore)\leq (\tau-\zeta^{1/2})\log R)=o(1)\text{.}\]
\end{proof}

\section{Proof of Theorem~\ref{thm:almostultrasmall}}

Here, we consider the case where $\alpha >1$. In this case, the main result in~\cite{bode_fountoulakis_mueller} implies that 
all components contain at most sublinear number of vertices. More precisely, we show that a.a.s. all components contain at 
most $N^{1/\alpha}$ vertices (up to a poly-logarithmic factor). In fact, there are many components of polynomial size (as there are 
many vertices of polynomial degree which do not belong to the same component). 

To prove Theorem~\ref{thm:almostultrasmall}, for any given vertex we explore a path that in a certain sense traverses its component.
We show that almost all vertices are close to such a \textit{spanning path}, which itself is short.
This results in short distances for most pairs of vertices which belong to the same component.

Note that since $\alpha>1$, a.a.s. there are no vertices inside the disk of radius $R/2$, that is, there are no core vertices in the resulting 
random graph. 

Consider also a connected component $C$ of the graph  $G$ embedded
into $\D$: for each pair of vertices which is an edge we connect the corresponding points by the geodesic segment that joins them. 
Let $I(C)$ denote the image of the embedding inside $\D$ equipped with a system of polar coordinates. Let $\theta : I(C) \to [0,2\pi)$ be the projection of $I(C)$ on the $\theta$-coordinate. Since this mapping  is continuous and
$I(C)$ is a closed set, it follows that $\theta (I(C))$ is a closed interval. 
If $\theta (I(C)) \subset [0,2\pi]$ (that is, a strict subset of $[0,2\pi)$), then 
the pre-images of the two end points must be vertices of $C$. We denote them by $v_f(C)$ and $v_\ell(C)$ (of course, a.a.s. there are no other points 
in $I(C)$ that are mapped to the endpoints of the $\theta (I(C))$).


\begin{definition}
For a connected component $C$, 
if the vertices $v_f(C)$ and $v_\ell(C)$ exist, we call a path $P=v_1,\dots,v_\ell$ in $C$ a \emph{spanning path} of $C$. 
	
	An \emph{umbrella} $U$ with root vertex $v$ is a spanning path $P$ of the component of $v$ together with a path connecting $v$ to $P$.
	The \emph{width} of the umbrella $U$ is the maximum among the (graph) distances of $v$ from the two endpoints of the associated spanning path.
\end{definition}
	%
\begin{figure}[h]
\centering
\includegraphics[width=0.5\columnwidth]{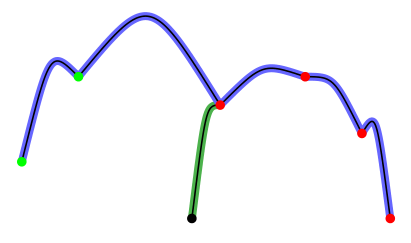}
	\caption{Example of an umbrella.}\label{fig:umbrella}
\end{figure}
Note that if $v',v''$ are two consecutive vertices of a spanning path and a vertex $u$ is contained in the sector of $\D$ defined by the angle $v'Ov''$ 
and, moreover, $t_u > \min \{t_u', t_{u''} \}$, then by Fact~\ref{fact:triangle} 
it is connected to one of them. 
Also, since there is no restriction on the length of the paths, if $v$ is on some spanning path $P$, then $P$ is an umbrella with root $v$.

The following follows immediately as the vertices of a component that are to the farthest in clockwise and anticlockwise direction are always
in a spanning path:

\begin{corollary}\label{cor:envMeet}
	If $P$ and $P'$ are spanning paths of the same component, then $P\cap P'\neq\emptyset$.
\end{corollary}

This fact allows us to do the following:
Given any pair of vertices $u$ and $v$ in the same component, construct a $u$-$v$-path by traversing an umbrella $U_u$ of $u$ until the
first vertex $z$ that is on an umbrella $U_v$ of $v$ is reached.
Then $uU_uzU_vv$ is a path connecting $u$ and $v$.
Thus the following lemma is key to the proof of Theorem~\ref{thm:almostultrasmall}.

\begin{lemma}\label{lem:umbrellaLength}
	Let $\eps>0$.
	For a vertex $v$ of $\Gnan$, a.a.s. there is an umbrella for $v$ in its component that has width at most $\log^{1+\eps}\log N$.
\end{lemma}

For the proof of this lemma we define the \textit{simultaneous breadth exploration process} starting at a vertex $v$ similar to the one
that we introduced in~\cite{bode_fountoulakis_mueller}.
Here, we keep track of two sets of vertices $V_\ell$ and $V_r$, which both start out as $\{v\}$.
Roughly speaking, we update the two sets adding the neighbours of the current sets that are located in the clockwise and anticlockwise
direction from the ``current" vertices, respectively.
If there are no neighbours that are farther in the clockwise direction of $V_r$ and no neighbours that are farther in the anticlockwise
direction of $V_\ell$, then the process stops.
%
We define the process starting at vertex $v$ as the following steps:
\begin{enumerate}[(i)]
	\item Let $V_\ell^{(0)}=V_r^{(0)}=\{v\}$ and let $i:=1$.
	\item Let $V_\ell^{(i)}$ be the set of vertices not in $V_\ell^{(i-1)}\cup V_r^{(i-1)}$ that are neighbours of some vertex 
		in $V_\ell^{(i-1)}\cup V_r^{(i-1)}$ and are in the clockwise direction of every vertex in 
    $\bigcup_{j=0}^{i-1}\{V_\ell^{(j)}\cup V_r^{(j)}\}$. 
    We define similarly the set $V_r^{(i)}$ as the set of vertices not in $V_\ell^{(i-1)}\cup V_r^{(i-1)}$ that  are neighbours of some vertex
    in $V_\ell^{(i-1)}\cup V_r^{(i-1)}$ and are in the anticlockwise direction of every vertex in 
   $\bigcup_{j=0}^{i-1}\{V_\ell^{(j)}\cup V_r^{(j)}\}$.
	\item If $V_\ell^{(i)}=\emptyset=V_r^{(i)}$, then stop.
		Otherwise, let $i:=i+1$ and go to step (ii).
\end{enumerate}
We call a repetition of steps (ii) and (iii) a \emph{round}.
To prove Lemma~\ref{lem:umbrellaLength}, we show that this process yields an umbrella and bound the number of steps needed until completion.

\begin{lemma}\label{lem:roundsUmbrella}
	If the simultaneous breadth exploration process starting at a vertex $v$ stops after $k$ rounds, then there is an umbrella for $v$
that has width at most $k$. 
\end{lemma}
\begin{proof}
	Let $C(v)$ denote the connected component that $v$ belongs to. 
    Let $V'_i=\bigcup_{j=0}^{i}\{V_\ell^{(j)}\cup V_r^{(j)}\}$, that is, the set of vertices discovered up to round $i$. 
	We denote by $v_\ell'$ the vertex in $V_i'$ with the largest relative angle with $v$ in the clockwise direction. 
We let $\theta_\ell^{(i)}$ be this angle and let $t_\ell^{(i)}$ be the type of this vertex. 
Similarly, let $v_r'$ be the vertex of $V_i'$ that is the farthest in the anticlockwise direction,
	and let $\theta_r^{(i)}$ and $t_r^{(i)}$ denote its angle and type. 
	Note that there is an edge between some vertex $v_\ell$ in $V'_{i-1}$ to the vertex $v_\ell'$ in $V_\ell^{(i)}$ and also an edge
    between some vertex $v_r \in V'_{i-1}$ and the vertex $v_r'$.
%
 	
	We now claim that if the process stops at round $k$, then the vertices $\hat v_r$ and $\hat v_\ell$ that are the farthest to 
    the anticlockwise and clockwise direction of $C(v)$ belong to $V'_{k-1}$.
	Note that $V_\ell^{(k)}=V_r^{(k)}=\emptyset$, so $V'_{k-1}=V'_{k}$.
	Assume this is not the case, so without loss of generality $\hat v_r\notin V'_{k-1}$.
	As $v$ and $\hat v_r$ are in the same component, there is a path $P$ from $v$ to $\hat v_r$.
	Let $w$ be the first vertex on $P$ that is outside the range of angles from $\theta_\ell^{(k-1)}$ to $\theta_r^{(k-1)}$. 
	Since $\hat v_r$ is the vertex that is farthest in the anticlockwise direction and $\hat v_r\notin V'_k$ this vertex must exist.
	Let $u$ be the predecessor of $w$ on $P$.
	We cannot have $u\in V'_k$ as otherwise $w$, being farther in the clockwise or anticlockwise direction than any other vertex in $V'_k$,
	must also be in $V'_k$ by the choice made in step (ii).
	There exists an $i<k$ and two \emph{adjacent} vertices $x$ and $y$ such that $x$ has been discovered at round $i-1$ and 
    $y$ has been discovered at round $i$ and $u$ is between $x$ and $y$. Now, if $t_u \geq t_y$, then by Claim~\ref{clm:pos} 
    ($x,u,y$ playing the role of $w,y,z$) it follows that $u$ is adjacent to $x$ as well. If $t_u < t_y$, then again Claim~\ref{clm:pos} 
    implies that $y$ is adjacent to $w$. 	
Hence, in either case $w$ would have been discovered by round $i+1$, whereby $w\in V_r^{(i+1)}\cup V_\ell^{i+1}\subseteq V'_k$; a
    contradiction.
	
	So both $\hat v_\ell$ and $\hat v_r$ are in $V'_k$.
	Note that every vertex in $V_\ell^{(i)}\cup V_r^{(i)}$ has a neighbour in $V_\ell^{(i-1)}\cup V_r^{(i-1)}$, so we can find a paths
	$P_\ell$ and $P_r$ of length at most $k$ from $\hat v_\ell$ to $v$ and from $\hat v_r$ to $v$, respectively.
	Together, possibly deleting redundant subpaths in $v_\ell P_\ell vP_rv_r$, we have an umbrella for $v$ of width at most $k$.
\end{proof}

We are now ready to prove Lemma~\ref{lem:umbrellaLength}
\begin{proof}[Proof of Lemma~\ref{lem:umbrellaLength}]

	We aim to bound the number of rounds it takes for the simultaneous breadth exploration process started at some vertex $v$ to stop.
	By Corollary~\ref{lem:innerEmpty}, it would be sufficient to consider a variation of the simultaneous breadth exploration process 
	where we expose only those vertices that have type at most $R/(2\alpha) + \omega (N)$, for some slowly growing function 
    $\omega (N) \rightarrow \infty$. We will use the same notation for the parameters of the process as in the unmodified process. 

    Let $T$ denote the stopping time of this process.
	Without loss of generality, assume that $V_\ell^{(i)},V_r^{(i)}\neq\emptyset$ for $i=1,\dots,T-1$.
	Define $V'_i$, $\theta_\ell^{(i)}$ and $\theta_r^{(i)}$ as in the previous proof (but for the modified process).
	Unlike the last proof, let $t_\ell^{(i)}$ and $t_r^{(i)}$ be the maximum types of vertices in $V_\ell^{(i)}$ and $V_r^{(i)}$, respectively,
	and they are set  to $0$, if the corresponding set contains no vertices.
	Let $t_i=\max\{t_\ell^{(i)},t_r^{(i)}\}$. Let $p_\ell^{(i)}$ be the point of type $t_i$ and angle $\theta_\ell^{(i)}$ in the 
	clockwise direction from $v$. Similarly, let $p_r^{(i)}$ be the point of type $t_i$ and angle $\theta_r^{(i)}$ in the 
	anticlockwise direction from $v$.  
	\begin{claim} \label{clm:belonging}
	We have $V_\ell^{(i+1)} \subset T_\eps^+(p_\ell^{(i)})$ and $V_r^{(i+1)} \subset T_{\eps}^+ (p_r^{(i)})$. 
	\end{claim}
	\begin{proof}[Proof of Claim~\ref{clm:belonging}]  
	Let $p$ be a point that is within hyperbolic distance $R$ from $u \in V_\ell^{(i)} \cup V_r^{(i)}$ and 
    satisfies $\vartheta_{p,v} > \theta_{\ell}^{(i)}$. 
    Let $u'$ be the point of type $t_\ell^{(i)}$, which has $\theta_{u,u'}=0$. 
    
    Note that $\vartheta_{p,p_\ell^{(i)}} \leq \vartheta_{p,u}$. 
    Since $p \in T_\eps^+(u)$, we have $\vartheta_{p,u} \leq 2(1+\eps)\frac{\nu}{N} e^{\frac{t_p+t_u}{2}}$. 
    As $t_u \leq t_{u'}=t_\ell^{(i)}$, it follows that $\vartheta_{p,u}\leq 2(1+\eps ) \frac{\nu}{N} e^{\frac{t_p+t_\ell^{(i)}}{2}}$. 
    In other words,  $p \in T_{\eps}^+(p_\ell^{(i)})$. Thereby, $V_\ell^{(i+1)} \subset T_{\eps}^+ (p_\ell^{(i)})$.   
  
    The proof that $V_r^{(i+1)} \subset T_{\eps}^+ (p_r^{(i)})$ is analogous.    
	\end{proof}

The above claim implies that the highest type of a vertex in $V_\ell^{(i+1)}$, which we denoted by $t_\ell^{(i)}$, 
is stochastically dominated by the highest type among the vertices in
$\left\{p \in T_\eps^+ ( p_{\ell}^{(i)}) \ : \ \vartheta_{p,p_{\ell}^{(i)}}>0, \ t_p < R/(2\alpha) + \omega(N)    \right\}$ .
Similarly, the highest type of a vertex in $V_r^{(i+1)}$, which we denoted by $t_r^{(i)}$ 
is stochastically dominated by the highest type among the vertices in
$\left\{p \in T_\eps^+ ( p_{r}^{(i)}) \ : \ \vartheta_{p,p_{r}^{(i)}} < 0, \ t_p < R/(2\alpha) + \omega(N)       \right\}$ .
Let $ T_\ell (p_{\ell}^{(i)} )$  and $T_r (p_r^{(i)} )$ denote these two sets.

Thus, $t_{i+1}$ is stochastically bounded from above by the largest type in $T_\ell (p_{\ell}^{(i)}) \cup T_r (p_r^{(i)})$. 
In turn, this is stochastically bounded from above by the maximum type of a vertex in $T_\ell (p^{(i)}) \cup T_r ( p^{(i)})$ for a point
$p^{(i)}$ of type $t_i = \max \{t_\ell^{(i)},t_r^{(i)}\}$. 
We shall proceed with the estimation of the cdf of the latter random variable. 

Observe first that Claim~\ref{clm:belonging} implies that for all $0 < i \leq T$ we have 
$V_i' \subset \bigcup_{j=0}^{i-1} \{ T_\eps^+ (p_\ell^{(j)}) \cup T_\eps^+ (p_r^{(j)}) \}$, 
assuming that $p_\ell^{(0)},p_r^{(0)}$ are both set to the point of $\D$ where $v$ is located.  
Let $\mathcal{N}_i$ be the set of vertices that belong to $V_i'$. For a vertex $u \in \V \setminus V_i'$,
the distribution on $\D$ is uniform (within the plane of curvature $-\alpha^2$) on the subset of $\D$ that excludes the union of the 
balls of radius $R$ around each vertex in $V_i'$. Recall that $\mu_{\alpha} (\cdot )$ is defined in \eqref{mu(U)Definition}. 
By Lemma~\ref{lem:relAngle} and the above observation, the area of the latter is at 
most $\sum_{j=0}^{i-1} \mu_{\alpha} (T_\eps^+ (p_\ell^{(j)}) \cup T_\eps^+ (p_r^{(j)}))$.  
But for each $j$, the angle that is spanned by $T_\eps^+ (p_\ell^{(j)}) \cup T_\eps^+ (p_r^{(j)})$ is proportional to  
$e^{R/(2\alpha) - R + \omega (N)}=o(1)$. Thus, if $i< R$, then we have
$\sum_{j=0}^{i-1} \mu_{\alpha} (T_\eps^+ (p_\ell^{(j)}) \cup T_\eps^+ (p_r^{(j)})) = o (\mu_{\alpha}(\D) )$. 

Using this, we conclude that the conditional probability that a vertex $u \in \V \setminus \mathcal{N}_i$ belongs to 
$T_\eps^+ (p^{(i)})$ and has type $t_u$ that satisfies 
$t\leq t_u < R/(2\alpha) +\omega(N)$ is at most 
\begin{equation*}
\begin{split}
&\int_{t}^{\frac{R}{2\alpha} +\omega(N)}\frac{4(1+\eps)}{2\pi}e^{\frac{t_i+t'-R}{2}}
\frac{\alpha \sinh (\alpha (R-t'))}{\cosh (\alpha R) (1-o(1))}dt' \\ 
&\leq \frac{2 \alpha (1+2\eps)}{\pi} e^{\frac{t_i -R}{2}} \int_{t}^{\frac{R}{2\alpha} +\omega(N)}
e^{t'/2} \frac{e^{\alpha (R-t')}}{2\cosh (\alpha R) (1-o(1))}dt' \\
&\leq  \frac{2 \alpha (1+3\eps)}{\pi} e^{\frac{t_i -R}{2}} 
\int_{t}^{\frac{R}{2\alpha} +\omega(N)}
e^{\left(\frac{1}{2} - \alpha \right) t'} dt' \\
&=\frac{2 \alpha \nu (1+3\eps)}{\pi} \frac{e^{t_i/2}}{N}
\int_{t}^{\frac{R}{2\alpha} +\omega(N)}
e^{\left(\frac{1}{2} - \alpha \right) t'} dt' < \frac{4 \alpha \nu (1+3\eps)}{\pi (2\alpha -1)} \frac{e^{t_i/2}}{N} 
e^{\left(\frac{1}{2} - \alpha \right) t},
\end{split}
\end{equation*}
for $N$ sufficiently large. 
Therefrom, the conditional probability that \emph{none} of the vertices in $\V \setminus \mathcal{N}_i$ satisfies this is at least 
\begin{equation} \label{eq:prb>t}
\begin{split}
& \left(1-  \frac{4 \alpha \nu (1+3\eps)}{\pi (2\alpha -1)} \frac{e^{t_i/2}}{N} 
e^{\left(\frac{1}{2} - \alpha \right) t} \right)^{|\V \setminus \mathcal{N}_i|} > 
\left( 1-  \frac{4 \alpha \nu (1+3\eps)}{\pi (2\alpha -1)} \frac{e^{t_i}/2}{N} 
e^{\left(\frac{1}{2} - \alpha \right) t} \right)^{N} \\
&> \exp \left( - D_{\alpha,\nu,\eps} e^{\frac{t_i}{2} - (\alpha -1/2)t} \right),
\end{split}
\end{equation}
for some $D_{\alpha,\nu,\eps} >0$ and any $N$ sufficiently large. 

Therefore, for $i<R$ the random variable $\max \{ t_\ell^{(i+1)}, t_r^{(i+1)}\}$ conditional on the history of the process up to step $i$
is stochastically dominated by a random
variable that follows the Gumbel distribution. The expectation of the latter is 
$$ \frac{t_i + 2\ln (2D_{\alpha,\nu,\eps})}{2\alpha -1} + \frac{2\gamma}{2\alpha -1}, $$
where $\gamma$ is Euler's constant. 
 Therefore, the following inequality holds:
 $$ \Ex{t_{i+1} | \mathcal{F}_i} \leq  \frac{t_i + 2\ln (2D_{\alpha,\nu,\eps})}{2\alpha -1} + \frac{2\gamma}{2\alpha -1}, $$
 where $\mathcal{F}_i$ denotes the sub-$\sigma$-algebra generated by the process up to step $i$. 
There exists a constant $U_{\alpha,\nu,\eps}>0$ such that when $t_i > U_{\alpha,\nu,\eps}$, we have 
\begin{equation} \label{eq:Supermart} 
\Ex{ t_{i+1}| \mathcal{F}_i } \leq  \frac{t_i + 2\ln (2D_{\alpha,\nu,\eps})}{2\alpha -1} + \frac{2\gamma}{2\alpha -1} < 
\frac{\alpha}{ 2\alpha -1} t_i =: \lambda_\alpha t_i< t_i.
\end{equation}

On the other hand, (\ref{eq:prb>t}) implies that if $t_i \leq U_{\alpha, \nu, \eps}$, then 
\begin{equation}\label{eq:lowInterval} \P ( t_{i+1}=0) \geq p>0, \end{equation}
for some positive constant $p$. 

With these tools, we can bound the stopping time $T$ of the process. 
Let $[T_1^{(s)}, T_2^{(s)} \wedge R]$ denote the $s$th interval of indices in which the process stays above $U_{\alpha, \nu, \eps}$.  
By (\ref{eq:Supermart}), for $T_1^{(s)} < i \leq T_2^{(s)} \wedge R$ the process $(t_i)$ is a supermartingale with decay rate at most $\lambda_{\alpha}$.  
\begin{claim} \label{clm:length_int}
For any $\eps' >0$
$$\Pr ( (T_2^{(s)} \wedge R) - T_1^{(s)}  \geq \log_{1/\lambda_{\alpha}}^{1+\eps'} R)= o(1).$$
\end{claim}
\begin{proof}[Proof of Claim~\ref{clm:length_int}]
Let $S:=\log_{1/\lambda_{\alpha}}^{1+\eps'} R$ and let $T^{(s)} :=T_2^{(s)} \wedge R$. 
Note that $\Ex { t_{i \wedge T^{(s)}} \mid \mathcal{F}_{T_1^{(s)}} } 
\leq \lambda_\alpha^{i\wedge T^{(s)} - T_1^{(s)}} t_{T_1^{(s)}} \leq 
\lambda_\alpha^{i \wedge T^{(s)} - T_1^{(s)}} R$. 
Let $A$ be the event $\{ T^{(s)} > S +T_1^{(s)}\}$. 
If $\omega \in A$, then
$\lambda_\alpha^{(S + T_1^{(s)}(\omega))\wedge T^{(s)}(\omega ) - T_1^{(s)}} t_{T_1^{(s)}}(\omega )  < \lambda_{\alpha}^S R=o(1)$.
By the definition of the conditional expectation,
we deduce that $ \Ex{ t_{(S + T_1^{(s)}) \wedge T^{(s)}} \mathbf{1}_{A}} =o(1)$ and since 
$\Ex{ t_{(S + T_1^{(s)}) \wedge T^{(s)}} \mathbf{1}_{A}} > U_{\alpha,\nu , \eps} \Pr (A)$, we finally deduce that $\Pr (A) = o(1)$.  
\end{proof}
Now, the length of the (discrete) interval $(T_2^{(s)}, T_1^{(s+1)} \wedge T \wedge R)$ is stochastically bounded from above 
by a geometric random variable that has parameter at least $p$.  

We call the union of these intervals an \emph{epoch}, that is, we call an epoch the interval 
$[T_1^{(s)}, T_1^{(s+1)}\wedge T \wedge R)$, for some 
$s>0$. 
By the above claim, for any $\eps' >0$, with probability $1-o(1)$, we have 
$(T_2^{(s)}\wedge R) - T_1^{(s)} \leq \log_{1/\lambda_\alpha}^{1+\eps'} R$. Additionally, the stochastic upper bound on the interval $(T_2^{(s)},T_1^{(s+1)})$ implies that this is at most $\log_{1/\lambda_\alpha}^{\eps'} R$ with probability $1-o(1)$. Hence, 
with probability $1-o(1)$ an epoch lasts for at most $\log_{1/\lambda_\alpha}^{1+2\eps'} R$ steps. Finally, since every epoch has probability 
at least $p$ to be the final one, it follows that the process hits $0$ within  $\log_{1/\lambda_\alpha}^{1+3\eps'} R$ steps with 
probability $1-o(1)$. In other words, a.a.s. we have $T \leq \log_{1/\lambda_\alpha}^{1+3\eps'} R$.

\end{proof}

Using the previous lemmas we prove Theorem~\ref{thm:almostultrasmall}.
\begin{proof}[Proof of Theorem~\ref{thm:almostultrasmall}]
	Let $0<\eps'<\eps$.
	Let $V'$ be the set of vertices in $\Gnan$ that have an umbrella of width at most $\log^{1+\eps'}\log N$.
	By Lemma~\ref{lem:umbrellaLength} we have $|V'|=(1-o(1))N$ a.a.s.
	For any $u,v\in V'$, if they are in the same component, by Corollary~\ref{cor:envMeet} the umbrellas are not disjoint.
	Thus there is a $u$-$v$-path of length at most $|U_u|+|U_v|\leq2\log^{1+\eps'}\log N<\log^{1+\eps}\log N$ for $N$ large enough.
\end{proof}

\bibliographystyle{apalike}

\end{document}